\documentclass[10pt]{amsart}

\usepackage{amsfonts,amsmath,amssymb,psfrag,graphicx,amsthm}
\usepackage{epsfig}

\newtheorem{theorem}{Theorem}
\newtheorem{definition}[theorem]{Definition}
\newtheorem*{definition*}{Definition}

\newtheorem{lemma}[theorem]{Lemma}
\newtheorem{proposition}[theorem]{Proposition}
\newtheorem{corollary}[theorem]{Corollary}
\newtheorem{conjecture}[theorem]{Conjecture}
\newtheorem{assum}[theorem]{Assumption}
\newtheorem{example}[theorem]{Example}
\newtheorem{remark}[theorem]{Remark}
\newtheorem*{remark*}{Remark}

\def\la{\langle}
\def\ra{\rangle}

\newcounter{notas}

\newcommand{\epsh}[2]
         {\begin{array}{c} \hspace{-1.3mm}
        \raisebox{-4pt}{\epsfig{figure=#1,height=#2}}
        \hspace{-1.9mm}\end{array}}

\newcommand{\qbin}[3]{\left[\begin{array}{c}
#1 \\
#2 \end{array}\right]_{#3}}

\def\vol{\mathrm{Vol}}
\def\c{c}
\def\D{{\mathbb{D}}}
\def\QM{{\mathbb{V}}}

\def\V{V}
\def\E{E}
\def\Z{\mathbb{Z}}

\def\ev{{\rm ev}}

\def\Ga{\Gamma}

\def\ga{\gamma}
\def\C{\mathbb{C}}
\def\R{\mathbb{R}}

\def\D{\mathbb{D}}
\def\P{\mathcal{P}}
\def\Ptr{P^{\mathrm{trunc}}}

\def\H{{\mathbb{H}}}

\def\mr{{\mathbb{R}}}

\def\mz{\mathbb{Z}}
\def\mn{\mathbb{N}}

\def\nns{\negthickspace}
\def\Hol{\mathcal Hol}
\def\I{\mathbf i}

\newcommand{\brkauf}[2]{\langle #1, #2\rangle}

\newcommand{\brunit}[2]{\langle #1, #2\rangle ^{{\rm U}}}

\newcommand{\smallunknot}[1]{\epsh{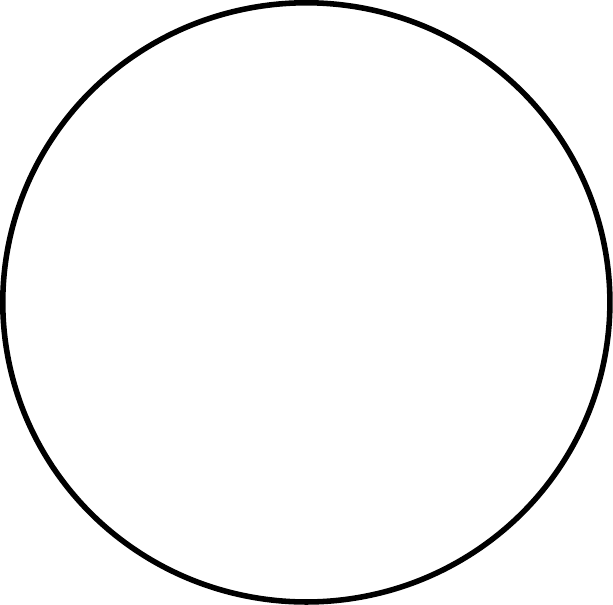}{#1}}

\newcommand{\smalltetra}[1]{\epsh{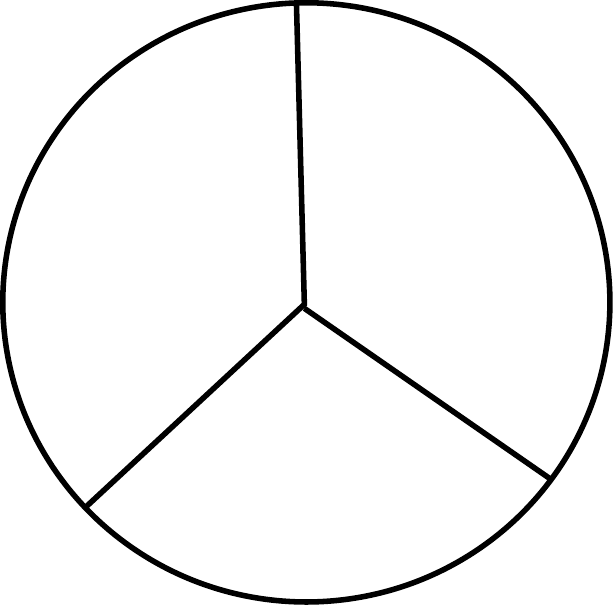}{#1}}
\newcommand{\smallcrossedtetra}[1]{\epsh{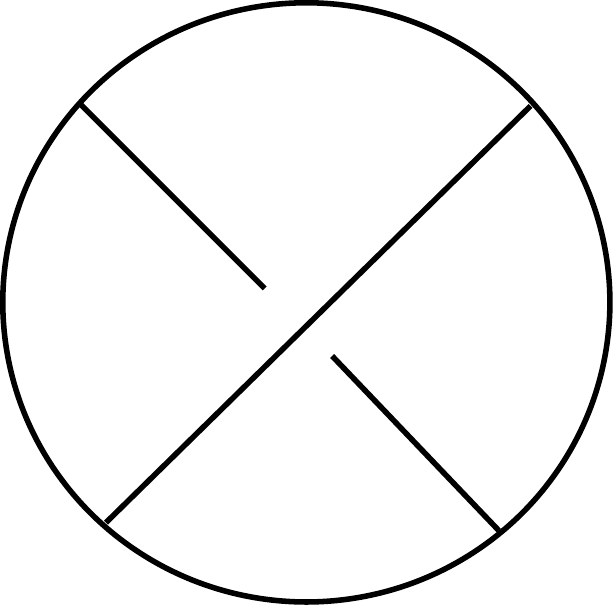}{#1}}

\title{On the volume conjecture for polyhedra}
\author{Francesco Costantino, Francois Gu\'eritaud and Roland van der Veen}

\begin{document}

\begin{abstract}
We formulate a generalization of the volume conjecture for planar graphs. Denoting by $\la \Ga, c \ra^{\mathrm{U}}$ the Kauffman bracket of the graph $\Ga$ whose edges are decorated by real ``colors'' $c$,  the conjecture states that, under suitable conditions, certain evaluations of $\la \Ga,\lfloor kc \rfloor \ra^{\mathrm{U}}$ grow exponentially as $k\to\infty$ and the growth rate is the volume of a truncated hyperbolic hyperideal polyhedron whose one-skeleton is $\Gamma$ (up to a local modification around all the vertices) and with dihedral angles given by $c$.
We provide evidence for it, by deriving a system of recursions for the Kauffman brackets of planar graphs, generalizing the Gordon-Schulten recursion for the quantum $6j$-symbols. Assuming that $\la \Ga,\lfloor kc \rfloor \ra^{\mathrm{U}}$ does grow exponentially these recursions provide differential equations for the growth rate, which are indeed satisfied by the volume (the Schl\"afli equation); moreover, any small perturbation of the volume function that is still a solution to these equations, is a perturbation by an additive constant. In the appendix we also provide a proof outlined elsewhere of the conjecture for an infinite family of planar graphs including the tetrahedra. 
\end{abstract}

\maketitle

\numberwithin{theorem}{section}

\tableofcontents

\section{Introduction and basic definitions} \label{intro}
\subsection{The Volume Conjecture}
The Volume Conjecture, initially formulated by R. Kashaev \cite{K} and then recast in terms of evaluations of colored Jones polynomials by Murakami and Murakami \cite{MM}, states that if $k\subset S^3$ is a hyperbolic knot, and if $J_n(A)\in \mz(A^{\pm1})$ is the ${n}^{\mathrm{th}}$ colored Jones polynomial of the knot, normalized so that its value on the unknot is $1$, then the following holds:
$$\lim_{n\to \infty} \frac{2\pi}{n} \left | \log \left ( J_n \left ( \exp{\frac{\I\pi}{2n}} \right ) \right ) \right |=\vol(S^3\setminus k).$$
Here $\I\in\mathbb{C}$ is a square root of $-1$, chosen once and for all.

The conjecture in the above form has been formally checked for the Figure Eight knot \cite{MM}, for torus knots \cite{KT} and their Whitehead doubles \cite{Zh}. It was also generalized to links and verified for the Borromean link \cite{MM3} and later for infinite families of links and knotted graphs in \cite{Va,Va2}. Moreover, there is experimental evidence of its validity for the knots $6_3$, $8_9$ and $8_{20}$ \cite{MMOTY}. In \cite{CoGen}, an extension of the conjecture was proposed to include the case of links in connected sums of $S^2\times S^1$; this extension was proved for the infinite family of ``fundamental shadow links'' contained in connected sums of copies of $S^2\times S^1$. In \cite{CoBP}, a further extension of the conjecture was proposed for links in arbitrary manifolds.
Recently another extension was proposed in \cite{Munew}.

One of the main difficulties of the conjecture is that in general it is difficult to understand in sufficient detail the rigid geometric structure of the knot complement. To avoid this problem, the first author formulated \cite{CoNY} a version of the conjecture for planar trivalent graphs and outlined a proof of the conjecture for an infinite class of cases. 
The intersection of the family of planar graphs with that of knots consists of the unknot, so the new conjecture is actually a generalization in a new direction, in which the topology of the complement of the graphs is easy (it is always a handlebody) and the geometry is rich (polyhedra can deform) without being inaccessible. Indeed we show that under suitable combinatorial conditions the asymptotical behavior of the invariants we consider is related to the hyperbolic volume of the polyhedra whose one-skeleton is described by the graphs. 
Since a graph has more than one edge, colored Jones polynomials are no longer sufficient to describe the full deformation space of these polyhedra; so our conjecture deals with the natural generalization of these invariants to graphs, namely the Kauffman brackets (also known as quantum spin networks).

\subsection{Framework, notation, and results}\label{sub:conjecture}

\subsubsection{On hyperideal hyperbolic polyhedra}\label{sub:introgeom}
Let $\mathbb{K}^3\subset \mathbb{R}^3$ be the open ball representing the hyperbolic $3$-space via the Klein model.
Let $\P\subset \mathbb{R}^3$ be a finite convex Euclidean polyhedron such that each vertex 
$w_i$ of $\P$ lies in $\mathbb{R}^3\setminus \mathbb{K}^3$ and each edge of $\P$ intersects $\partial \mathbb{K}^3$. Let $C_i$ be the cone from $w_i$ tangent to $\partial\mathbb{K}^3$ and $\pi_i$ be the half-space not 
containing $w_i$ such that $\partial\pi_i\cap \partial \mathbb{K}^3=C_i\cap \partial \mathbb{K}^3$. 
In all the paper we will assume that  the $1$-skeleton of $\P$, denoted ${\P}^{(1)}$, is a trivalent graph.
\begin{definition*}[Truncated hyperbolic hyperideal polyhedra]
Given $\P$ as above the  associated \emph{truncated hyperideal hyperbolic polyhedron} is $\Ptr:={\mathcal P}\cap \bigcap_i \pi_i$ (see Figure \ref{fig:exampleP}).
We will say that a planar graph $\Ga\subset S^2$ is the $1$-skeleton of $\P$ (or, with an abuse of notation, of $\Ptr$), if there is a homeomorphism mapping $(S^2,\Ga)$ into $({\P}^{(2)},{\P}^{(1)})$. 
\end{definition*} 
\begin{figure}
$\epsh{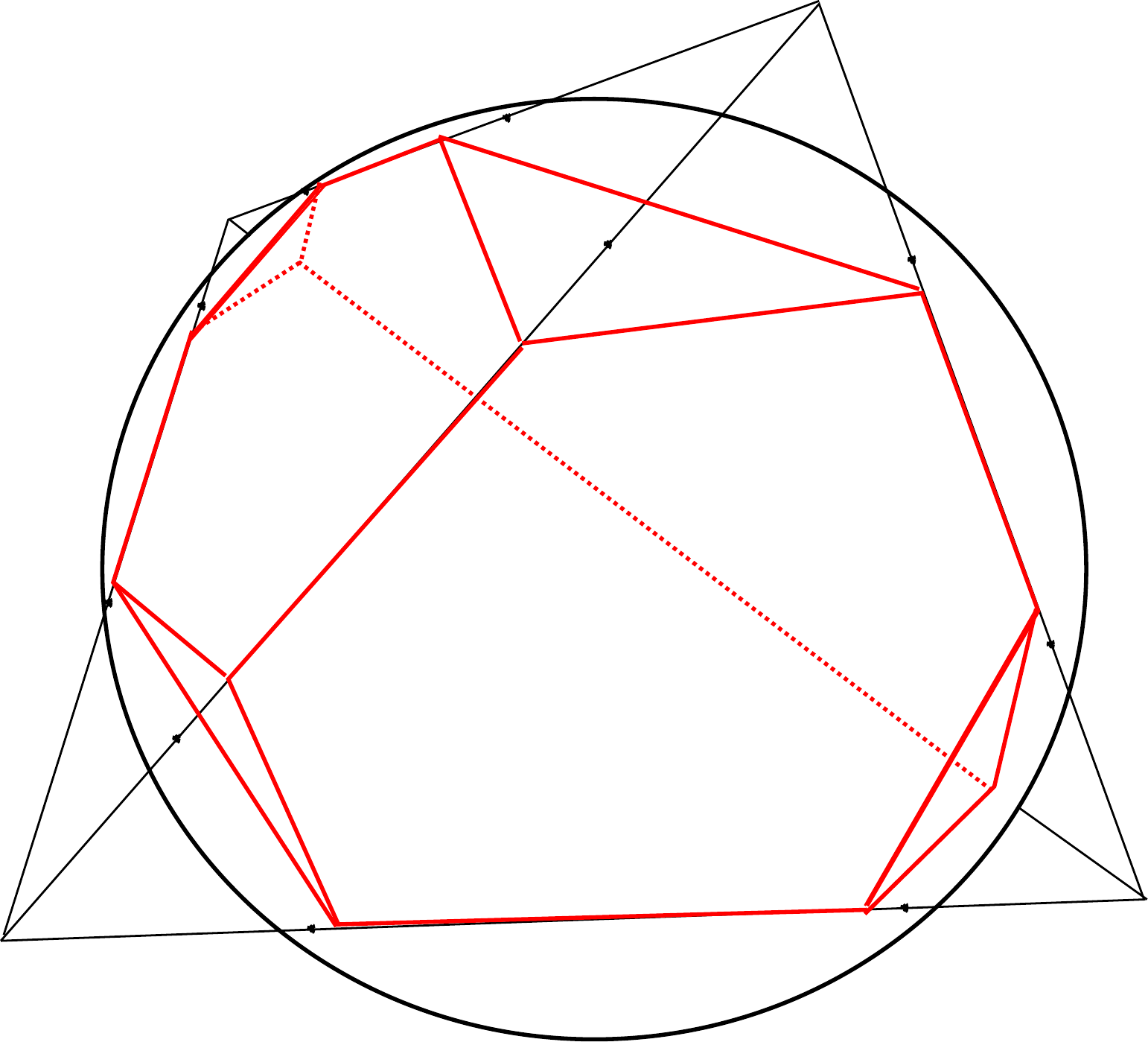}{20ex}
$
\caption{An example of $\Ptr$ (in this case, in red, a truncated tetrahedron): remark that $\Ptr$ is compact unless one of the edges of $\P$ is tangent to $\partial  \mathbb{K}^3$.}\label{fig:exampleP}
\end{figure}

(Observe that $(\Ptr)^{(1)}$ is homeomorphic to the graph obtained by replacing each vertex of  ${\P}^{(1)}$ by a triangle as follows: \raisebox{-0.1cm}{\includegraphics[width=1.0cm]{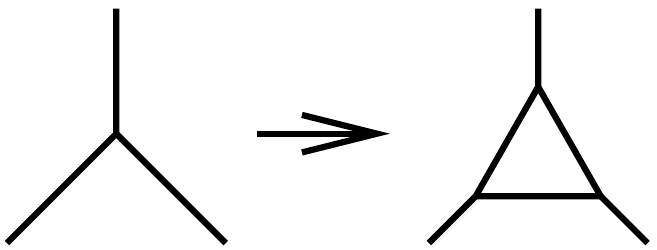}}.)

A geometric structure on a truncated hyperideal polyhedron $\Ptr$ is uniquely identified by the exterior dihedral angles at the edges of $\P$ hence, if $\Gamma=(\Ptr)^{(1)}$, by a map $\gamma:E(\Gamma)\to (0,\pi]$, where $E(\Ga)$ is the set of edges of $\Ga$. We denote $A(\P)$ the set of all possible such angle structures; by a theorem of Bao and Bonahon \cite{BaBo} (see Section \ref{sec:polyhedra} for more details), $A(\P)$ is a subset of $(0,\pi]^E$ cut out by a finite list of linear inequalities. (We are including the dihedral angle $\pi$ which corresponds to the case when an edge of $\P$ is tangent to $\partial \mathbb{K}^3$, in which case the corresponding edge in $\Ptr$ reduces to an ideal point.) To specify a geometric structure on $\Ptr$ we will write $\Ptr(\gamma)$ for some $\gamma \in A(\P)$.  

If each edge of ${\P}$ intersects $\partial \mathbb{K}^3$ in $2$ points (equivalently if $\gamma(e)\in (0,\pi)$ for all edges) it is easy to check that $\Ptr(\gamma)$ is a compact hyperbolic polyhedron whose edges are of two kinds: those contained in the \emph{truncation faces} (i.e.\ in $\cup_i (\partial \pi_i)\cap \P$) and the remaining ones which are contained in the edges of $\P$. The latter edges of $\Ptr(\gamma)$ have well-defined finite lenghts $\ell_i$ and exterior dihedral angles given by~$\gamma_i$. 
Edges that are reduced to ideal points (with $\gamma_i=\pi$) have length $\ell_i=0$.
We will denote by $\vol:A(P)\to \mr$ the function that associates to each set of dihedral exterior angles $\gamma\in A(\P)$ the hyperbolic volume of the corresponding $\Ptr$. It is well known that $\vol:A(\P)\to \mathbb{R}$ is smooth and satisfies a differential equation called the \emph{Schl\"afli formula}, stating that 
$\mathrm{d}\,  \vol=\frac{1}{2}\sum_{i} \ell_i \mathrm{d}\gamma_i$. 
In particular, 
\begin{equation} \label{eq:schlaefli} \ell_i=2\frac{\partial \vol(\gamma)}{\partial \gamma_i}~. \end{equation}

Here is a paraphrase of the result we prove in Section \ref{sec:faceequations}:
\begin{theorem} \label{prop:faceequation}
\begin{enumerate}
Let $\Ptr$ be a truncated hyperideal polyhedron. 
\item[a.] For each face of $\Ptr$ there is an $\mathrm{SL}_2(\mathbb{R})$-valued equation satisfied by $\vol:A(\P)\to \mr$ of the form:\begin{equation} \label{eq:faceequation}
\prod_i \begin{pmatrix}
a_i(\gamma) & b_i(\gamma) \\
c_i(\gamma) & d_i(\gamma)
\end{pmatrix} \begin{pmatrix}
\cosh(\frac{\ell_i}{2}) & \sinh(\frac{\ell_i}{2}) \\
\sinh(\frac{\ell_i}{2}) & \cosh(\frac{\ell_i}{2}) 
\end{pmatrix}=-\mathrm{Id}
\end{equation}
where $i$ runs over all the edges of $\P$ contained in the face,  $\ell_i=2\frac{\partial \vol(\gamma)}{\partial \gamma_i}$ are the edge lengths of the interior edges of $\Ptr$ and $a_i,b_i,c_i,d_i$ are explicit smooth functions of the dihedral angles of $\mathcal P$ adjacent to the face. 

\item[b.] Moreover, if $f:U\to \mr$ is a function of class $C^1$ defined on an open connected subset $U$ of $A(\P)$, that is close enough to $\vol$ in the $C^1$ sense and satisfies all the $\mathrm{SL}_2(\R)$-valued equations \eqref{eq:faceequation} associated to the faces of $\P$, then $f-\vol$ is a constant.
\end{enumerate}
\end{theorem}
We stress that the definition of these equations is purely geometric and does not depend on any quantum object. The equations as in \eqref{eq:faceequation} can be seen via \eqref{eq:schlaefli} as polynomial equations relating the exponentiated dihedral angles and exponentiated edge lengths of $\Ptr$, or alternatively, as partial differential equations satisfied by the volume function.
 
\subsubsection{Invariants of framed graphs, statement of the conjecture and the main result}
We turn now to the statement of the volume conjecture for polyhedra. 
\begin{definition*}[KTG] \label{def:ktg}
A $Knotted\ Trivalent\ Graph$ (KTG) is a finite trivalent graph $\Ga\subset S^3$ equipped with a ``framing'', i.e.
an oriented surface retracting to $\Gamma$ (seen up to isotopy fixing $\Gamma$).
We denote by $\V$ the set of \emph{vertices} of $\Ga$ and $\E$ the set of \emph{edges}. 
\end{definition*}

\begin{definition}[Admissible coloring]\label{def:admcol}
A \emph{coloring} of $\Ga$ is a map $\c:E\to \frac{\mathbb{N}}{2}$ (whose values are called \emph{colors}); it is \emph{admissible} if for all $v\in V$, if $e_i,e_j,e_k\in E$ are the edges touching $v$, the following conditions are satisfied:
\begin{enumerate}
\item[a.] $\c(e_i)+\c(e_j)\geq \c(e_k),\ \c(e_j)+\c(e_k)\geq \c(e_i),\ \c(e_k)+\c(e_i)\geq \c(e_j)$
\item[b.] $\c(e_i)+\c(e_j)+\c(e_k)\in \mathbb{N}.$
\end{enumerate}
\end{definition}

As we shall recall in Section \ref{sec:shadowstates}, given a KTG $\Ga$ equipped with a coloring $\c$ there is an associated invariant of isotopy of $(\Ga,\c)$, called the \emph{unitary Kauffman bracket} (or quantum spin network), denoted $\brunit{\Ga}{\c}$, valued in the space 
$\Hol$ of functions which are holomorphic on the interior of the unit disc $\D\subset \C$ and whose square is the restriction of a meromorphic function on $\C$ (see Remark \ref{rem:squareok}).
\begin{definition*}[Real coloring]
Given a sequence of colorings $(\c_n)_{n\in \mathbb{N}}$ on $\Gamma$ we define the \emph{real coloring}  $\overline{\ga}$ as the map $\overline{\ga}:E\to \mathbb{R}$ given by $\overline{\ga}(e)=\lim_{n\to \infty} \frac{\c_n(e)}{n}$. We will always tacitly assume that these limits exist for all $e\in E$, for all the sequences we shall consider. 
We will also let $\gamma:E\to \mathbb{R}$ be defined in terms of $\overline{\gamma}$ as $$\gamma(e)=2\pi (1-\overline{\ga}(e)),\: \forall e\in E.$$
\end{definition*}
\begin{remark*}
Clearly $\overline{\ga}$ satisfies weak triangular inequalities as in Definition~\ref{def:admcol}--{\textrm a}. Conversely, given a map $\overline{\ga}:E\to \mathbb{R}$ satisfying strict triangular inequalities around vertices, one can produce a sequence $(c_n({\ga}))_{n\in \mathbb{N}}$ of admissible colorings of $\Ga$ whose limit is $\overline{\ga}$ by setting e.g.\ $(\c_n({\ga}))(e):=\lfloor \overline{\ga}(e) n\rfloor$. Here and below we take $\lfloor x\rfloor$ to be the greatest integer less than or equal to $x$.
\end{remark*}

\begin{definition*}[Evaluation]
Let $f$ be the restriction to $\D$ of a meromorphic function on $\C$; we denote by $\ev_n(f)$ the first nonzero coefficient of a Laurent series expansion of $f$ around $A=\exp(\frac{\pi \I}{2n})$. 

More generally if $f$ is a holomorphic function on $\D$ such that $f^2$ extends meromorphically to $\mathbb{C}$, then we let $\ev_n(f):=\pm\sqrt{\ev_n(f^2)}$ (in what follows only the modulus of $\ev_n$ will be relevant to our computations).
\end{definition*}

\begin{conjecture}[Volume Conjecture for polyhedra]\label{conj:vc} 
Let $\Ga=\P^{(1)}$ and let $\ga:E\to (0,\pi]$ be the exterior dihedral angles of $\P$. Let also $\overline{\ga}:=1-\frac{\ga}{2\pi}$ and $(\c_n(\ga))_{n\in \mathbb{N}}$ be the sequence of integer colorings defined by $\c_n({\ga}) = \lfloor n\overline{\ga}\rfloor$. Then the following limit holds: 
\[
\lim_{n\to\infty}\frac{\pi}{n}\log \left | \ev_n\brunit{ \Ga}{\c_n({\ga})} \right | = \vol (\Ptr(\ga)).
\]
\label{volumeconjecture}
\end{conjecture}

The factor $\pi$ instead of the usual $2\pi$ in the original volume conjecture can be explained by viewing the polyhedron als
one half of the geometric decomposition of the complement of the planar graph. In the special case $\ga=\pi$ the above agrees with
the conjecture in \cite{Va2}.

The following result, announced by the first author in \cite{CoNY}, can be deduced from the results in \cite{Co} but we provide a proof in the Appendix for the sake of self-containedness:
\begin{theorem}\label{teo:easycase}
Let $\Gamma$ be a planar graph obtained from a tetrahedron by applying any finite sequence of the local moves \raisebox{-0.1cm}{\includegraphics[width=1.0cm]{vertextotriangle.pdf}}. Then $\Ga$ is the $1$-skeleton of a hyperbolic ideal polyhedron $P_\Ga$ and for each $\gamma\in A(P_\Ga)$ Conjecture \ref{conj:vc} holds true.
\end{theorem}

Unfortunately there are plenty of trivalent graphs which cannot be obtained as above from a tetrahedron: e.g.\ the $1$-skeleton of a cube.
The main result of this paper is to provide evidence for Conjecture \ref{conj:vc} in these other cases. 
Let $\Gamma=\P^{(1)}$ and $\ga\in A(\P)$; let also $(\c_n(\ga))_{n\in \mathbb{N}}$ be a sequence of colorings on $\Ga$ with $\lim_{n\to \infty} \frac{\c_n(\ga)}{n}=\overline{\gamma}=1-\frac{\ga}{2\pi}$. From now on we will make the following assumption:
\begin{assum}\label{assum:1}
There exists a function $f:A(\P)\to \R$ of class $C^1$ such that for each map $\delta:\E(\Ga)\to \{-\frac{1}{2},0,\frac{1}{2}\}$ satisfying the second condition of Definition \ref{def:admcol} the following limits hold:
$$\lim_{n\to \infty}\lim_{A\to \exp\left(\frac{\I\pi}{2n}\right)} \frac{\brunit{\Ga}{\c_n(\ga)+\delta}}{\brunit{\Ga}{\c_n(\ga)}}
=
\exp\left(\sum_{e_i\in \E(\Ga)} 2\delta(e_i)\frac{\partial f}{\partial \gamma_i} \right).$$
\end{assum}
The limit can be seen as a form of $C^1$ convergence of the growth rate of the evaluations to the function $f$.
Indeed, if $\ev_n\left(\brkauf{\Ga}{\c_n(\ga)}\right)\sim \exp(n f(\ga))$ for some regular function $f:A(\P)\to \R$ and if the asymptotic approximation is itself regular enough, then the above limit holds. 
(So, although exponential growth does not formally imply Assumption \ref{assum:1}, it does under suitable regularity hypotheses.)

So our main result below roughly says that if the growth of evaluations of quantum spin networks with underlying graph $\Ga$ 
is exponential and well-behaved (Assumption \ref{assum:1}), then the growth rate resembles the volume. More precisely we have:
\begin{theorem}\label{teo:main}
Let $\Gamma$, $\P$, $\ga$ and $\c_n(\ga)$ be as above and let $f$ be the function whose existence is supposed in Assumption \ref{assum:1}. Letting $\ell_i=2\frac{\partial f}{\partial \ga_i}$, then for each face of $\P$ equation \eqref{eq:faceequation} is satisfied. 
\end{theorem}

By Theorem \ref{prop:faceequation}--b, this is a strong indication towards the fact that the growth rate function $f$ differs from the volume by at most a constant. 
Of course one of the main questions left open by this work is whether Assumption \ref{assum:1} does hold, or at least under which geometric or topological conditions it does.

The proof of Theorem \ref{teo:main} is based on the construction of a set of explicit recursion relations which are satisfied by the invariants $\brkauf{\Ga}{\c_n}$. The existence of these recursions can be proved in general via the techniques used in \cite{GaLe}, but it is typically very hard to compute it. We show that for planar graphs there is an easy way of producing such recursions (see Proposition \ref{prop:geometricrecursion}); in the special case when $\Ga=\smalltetra{0.4cm}$  our recursions recover the well known Gordon-Schulten recursion for $6j$-symbols of $U_q(\mathfrak{sl}_2)$. Then we prove Theorem \ref{teo:main} by studying the asymptotical behavior of the coefficients of these equations, which, surprisingly enough, turn out to be strictly related to those of the geometric Equations (\ref{eq:faceequation}). 

\subsubsection{Plan of the paper}
Section \ref{sec:polyhedra} collects facts about hyperbolic polyhedra, including Theorem \ref{prop:faceequation}. Section \ref{sec:shadowstates} contains the definitions and main properties of the quantum invariants $\langle \Gamma,c\rangle^{\mathrm{U}}$. The recursion relations satisfied by $\langle \Gamma,c\rangle^{\mathrm{U}}$ are shown in Section \ref{sec:recursions}; these are used in Section \ref{sec:mainproof} to prove Theorem \ref{teo:main}. Finally in Appendix \ref{sec:easycase} we prove Theorem \ref{teo:easycase}. 

\subsubsection{Acknowledgements}
F.C.  and F.G. were supported by the Agence Nationale de la Recherche through the projects QuantumG\&T (ANR-08-JCJC-0114-01) and DiscGroup (ANR-11-BS01-013), ETTT (ANR-09-BLAN-0116-01) respectively. F.G. was also supported by the Labex CEMPI (ANR-11-LABX-0007-01). RvV thanks the Netherlands organization for scientific research (NWO).

\section{Geometry of hyperideal hyperbolic polyhedra}\label{sec:polyhedra}
Let $\P$ and $\Ptr$ be as in Subsection \ref{sub:introgeom} and let 
${\P}^*$ be the polyhedron dual to ${\P}$ and $E^*=E(\P^*)$ be its set of edges which is naturally in correspondence with $E({\P})$; in what follows we will implicitly use this correspondence to translate a map $\ga:E({\P})\to (0,\pi]$ to one $\ga^*:E({\P}^*)\to (0,\pi]$.
Let $\Ga$ be a planar graph and $\Ga^*$ be its dual. A result of Steinitz shows that $\Ga$ can be the one-skeleton of a convex polyhedron in $\mathbb{R}^3$ iff $\Ga^*$ is $3$-connected i.e.\ it cannot be disconnected by deleting $0,1$ or $2$ vertices. So we shall suppose from now on that $\Gamma$ is $3$-connected.

The following theorem, proved by X. Bao and F. Bonahon (Theorem 1 of \cite{BaBo}, see also \cite{Ri}), allows one to identify the set $A(\P)$ of possible angle structures on $\Ptr$:
\begin{theorem}\label{teo:bonahon}
A map $\ga:E({\P})\to (0,\pi]$ is the set of exterior dihedral angles on a hyperbolic hyperideal polyhedron combinatorially equivalent to $\P$ iff each simple closed curve $c$ in the $1$-skeleton of ${\P}^*$ satisfies $\sum_{e\in c} \ga^*(e)> 2\pi$,  and each path $c'$ composed of edges of ${\P}^*$ such that $\partial c'$ is contained in the closure of a single face of ${\P}^*$ but $c'$ is not, satisfies $\sum_{e\in c'} \ga^*(e)> \pi$, where the sums run over the edges of ${\P}^*$ contained in $c$ (resp. $c'$). Such a map $\gamma$ identifies $\Ptr$ uniquely up to isometry.
\end{theorem}

\begin{example}[The hyperideal tetrahedron]\label{ex:tetrahedron}
If ${\mathcal P}$ is a tetrahedron, then so is ${\mathcal P}^*$ and if the set of exterior dihedral angles on $P$ is $(\alpha,\beta,\gamma)$ (on three edges  adjacent to a vertex of ${\mathcal P}$) and $(\alpha',\beta',\gamma')$ on the respective opposite edges, then the conditions for closed curves provided in Theorem \ref{teo:bonahon} impose the following inequalities:
\begin{align}
\label{babo1} \alpha+\beta+\gamma> 2\pi,\ \alpha'+\beta+\gamma'> 2\pi, \alpha'+\beta'+\gamma> 2\pi, \alpha+\beta'+\gamma'> 2\pi,\\
\label{babo2} \alpha'+\alpha+\beta'+\beta>2\pi , \ \alpha'+\alpha+\gamma+\gamma'>2\pi,\ \beta'+\beta+\gamma+\gamma'>2\pi.
\end{align}
But since all the angles are in $(0,\pi]$, Condition \eqref{babo1} entails that the angles on any two adjacent edges sum to more than $\pi$, which implies both \eqref{babo2} and the condition on non-closed paths in Theorem \ref{teo:bonahon}.
One also has $\alpha+\beta-\gamma>0$ (and similarly for all triples of angles around a vertex). 
\end{example}

\subsection{Theorem \ref{prop:faceequation}--a: Geometric equations associated to faces of $\P$}\label{sec:faceequations}
Let $F\subset \H^2\subset \H^3$ be a compact, convex polygon with $2p$ edges and whose exterior angles are all $\frac{\pi}{2}$. Let $v_0,v_1,\ldots v_{2p}=v_0$ be its vertices and let $\ell_{j,j+1}={\rm length(}\overline{v_jv_{j+1}})$; let $C(\overrightarrow{v_jv_{j+1}}):=C(\ell_{j,j+1})$ where:
$$C(\ell):=\begin{pmatrix}
\cosh\frac{\ell}{2} & \sinh\frac{\ell}{2}\\
\sinh\frac{\ell}{2} & \cosh\frac{\ell}{2}\\
\end{pmatrix},\ \text{ and let }\ 
M:=\frac{1}{\sqrt{2}}\begin{pmatrix} 1 &-1\\ 1 & 1\\ \end{pmatrix}.$$
\begin{proposition}\label{prop:circlegeometric}
The following holds:
 $$\prod_{j=0}^{2p-1} MC(\overrightarrow{v_jv_{j+1}}):=MC(\overrightarrow{v_{2p-1}v_0})MC(\overrightarrow{v_{2p-2}v_{2p-1}}) \cdots MC(\overrightarrow{v_0,v_1})=-{\rm Id}.$$
\end{proposition}
\begin{proof}
In the upper-half plane model of $\H^2$, up to isometry, we may suppose that $v_0=\I$ and that the edge $v_0v_1$ is on the geodesic $L$ connecting $-1$ and $1$, that its points have negative real part and finally that all the other vertices $v_i$ of the polygon $F$  are contained in the unit disc.
The isometry represented by $C(\overrightarrow{v_0v_1})$ slides the edge along $L$ until $v_1=\I$, then the matrix $M$ applies a $-\frac{\pi}{2}$ rotation with center $\I$. This isometry has then moved $v_1$ in the same position as the initial position of $v_0$ and we can iterate the process until $v_0$ returns to its initial position. The identity expresses the fact that these isometries overall describe the identity map, the minus sign is due to the fact that we are working in $SL_2(\mr)$ thus a full turn lifts to $-{\rm Id}$.
\end{proof} 

Let us recall the second law of cosines for hyperbolic triangles:
\begin{proposition}[Second law of cosines]\label{prop:alkashi}
Let $\gamma_i,\ i=1,2,3$ be the exterior angles of a hyperbolic triangle and $\ell_i$ be the lengths of the edges opposite to $\gamma_i$. 
If $\{i,j,k\}=\{1,2,3\}$, then $$\cosh(\ell_i)=\frac{\cos(\gamma_j)\cos(\gamma_k)-\cos(\gamma_i)}{\sin(\gamma_j)\sin(\gamma_k)}.$$
\end{proposition}

Let now $\P$ be a convex trivalent polyhedron in $\mathbb{R}^3$ whose truncation gives a hyperbolic hyperideal polyhedron $\Ptr$, and let $\{v_s\}$ be the vertices of $\Ptr$. 
Let $\lambda^{\mathrm{trunc}}=v_0,v_1,\cdots, v_{2p}=v_0$ be the boundary of a face of $\Ptr$.
To each edge $\overrightarrow{v_s v_{s+1}}$ of $\lambda^{\mathrm{trunc}}$ we may associate the hyperbolic length $\ell_s:=\ell(\overrightarrow{v_s v_{s+1}})\in \mr_+$ of $\overrightarrow{v_sv_{s+1}}$ and an external dihedral angle $\gamma_s:=\gamma(\overrightarrow{v_s v_{s+1}})\in (0,\pi]$ which is defined as $\frac{\pi}{2}$ if $\overrightarrow{v_{s}v_{s+1}}$ is a truncation edge (i.e.\ if $s$ is odd) and else it is the external dihedral angle at the edge $\overline{v_{s} v_{s+1}}$ of $\Ptr$. Furthermore, for each truncation edge (odd $s$) we may define an ``opposite angle'' $\gamma'_s$ as the external dihedral angle opposite to $v_sv_{s+1}$ in the truncation triangle containing $v_sv_{s+1}$.

Then the following holds (finishing \ref{prop:faceequation}--a):
\begin{theorem}\label{teo:volsatisfies}
For each loop $\lambda\subset P^{(1)}$ which is the boundary of an oriented face,
\begin{equation*}
\prod_{j=0}^{p-1}MC\left({\rm arccosh}\frac{\cos(\gamma_{2j+2})\cos(\gamma_{2j})-\cos(\gamma'_{2j+1})}{\sin(\gamma_{2j})\sin(\gamma_{2j+2})}\right)MC(\ell_{2j,2j+1})=-{\rm Id.}
\end{equation*}
\end{theorem}
\begin{proof}
It is a direct consequence of Proposition \ref{prop:circlegeometric}, once one computes the length of $\overrightarrow{v_{2j+1}v_{2j+2}}$ via Proposition \ref{prop:alkashi}.
\end{proof}
\subsection{Theorem \ref{prop:faceequation}--b: Rigidity of functions satisfying face equations.}
Let $\Ga$ be the one-skeleton of a polyhedron ${\P}$ and $f:U\to \mr$ be a smooth function defined on a connected open subset $U$ of $A(\P)$.
We say that $f$ \emph{satisfies the matrix-valued face equations} if, setting $\ell_i:=2\frac{\partial f}{\partial \gamma_i}$, the $\mathrm{SL}_2(\R)$-valued equations described in Theorem \ref{teo:volsatisfies} are satisfied for all simple closed curves $\lambda^{\mathrm{trunc}}$ in the $1$-skeleton of $\Ptr$ formed by the boundaries of the non-truncation faces of $\Ptr$.
The Schl\"afli formula \eqref{eq:schlaefli} implies that $\vol$, in particular, satisfies the matrix-valued face equations.
The goal of this subsection is to check that any function $f$ close to $\vol$ in the $C^1$ sense that also satisfies the matrix-valued face equations, is equal to $\vol$ up to an additive constant.

This is a simple consequence of the rigidity of hyperideal polyhedra (Theorem~\ref{teo:bonahon}).
Indeed, if the partial derivatives of $f$ were different from those of $\vol$ at a given point $\gamma\in (0,\pi]^{E(\P)}$, then $f$ would define a small deformation $\P_f$ of the hyperideal polyhedron $\P$ (with the same angles $\gamma$ but different edge lengths), which is impossible.

More precisely, any function $f$ as above allows us to define an $(\mathbb{H}^3, \mathrm{Isom}(\mathbb{H}^3))$-structure on a neighborhood of the polyhedron $\Ptr\subset \mathbb{R}^3$ with dihedral angles $\gamma$, as follows. First, let $U_1$ and $U_2$ be tubular neighborhoods of the $1$- and $2$-skeleta of $\Ptr$ in $\mathbb{R}^3$, respectively. Endow $\Ptr$ with formal (real) edge lengths, given by $\ell_i=2\frac{\partial f}{\partial \gamma_i}$ at interior edges, and by Proposition \ref{prop:alkashi} (independently of $f$) at truncation edges. Let $\widetilde{U_1}$ denote the universal cover of $U_1$. The above length data is enough to define a developing map $\widetilde{\Phi}_f: \widetilde{U_1} \rightarrow \mathbb{H}^3$ respecting all angles, edge lengths and framings, in the sense that $\widetilde{\Phi}_f$ takes the lift of the loop around a $2p$-gonal face $F$ of $\Ptr$ to a path $F_f$ formed of $2p$ coplanar segments in $\mathbb{H}^3$, each orthogonal to the next. We may assume that $F_f$ lies in $\mathbb{H}^2\subset \mathbb{H}^3$.

The identity of Proposition \ref{prop:circlegeometric}, associated to the loop around the $p$-gonal face of $\P$ containing $F$, implies that $F_f$ closes up in $\mathbb{H}^2$, with the given edge lengths, to form a right-angled planar $2p$-gon (still denoted $F_f$). Since $F_f$ and $F_{\vol}$ have nearly the same edge lengths, it follows that $F_f$ bounds a \emph{convex} right-angled $2p$-gon of~$\mathbb{H}^2$. As for truncation triangles, they similarly close up under $\widetilde{\Phi}_f$ since $\widetilde{\Phi}_f$ assigns them the same angles and edge lengths as truncation triangles of $\Ptr$.
Therefore $\widetilde{\Phi}_f$ decends to a map $\Phi_f:U_1\rightarrow \mathbb{H}^3$ which can be extended to the universal cover of $U_2$. As $U_2$ is simply connected, we actually have a developing map $\Phi_f:U_2\rightarrow \mathbb{H}^3$. Since $\Phi_f$ and $\Phi_{\vol}$ are very close, it follows that $\Phi_f$ (like $\Phi_{\vol}$) defines a realization of $\partial \Ptr$ as the boundary of a \emph{convex} truncated hyperideal polyhedron $\Ptr_f$. By rigidity of convex polyhedra, $\Ptr_f$ and $\Ptr$ are actually isometric (via $\Phi_f$) and have in particular the same edge lengths, which means that $f$ and $\vol$ have the same partial derivatives at $\gamma$, for all real tuples $\gamma$. Hence $f-\vol$ is locally constant: Theorem \ref{prop:faceequation}--b is proved.

\section{Shadow-state formulation of $\brkauf {\Ga}{\c}$}\label{sec:shadowstates}
In this section we will recall the definition of quantum spin networks through their ``shadow-state formulas'' (as these formulas will be used later on). 
We now fix a diagram $D$ of $\Ga$ such that the blackboard framing coincides with that of $\Ga$ (it is easy to check that it exists), and an admissible coloring $\c$ on $\Ga$, and recall how the Kauffmann bracket $\brkauf{\Ga}{\c}$ is defined.
Let $A\in\C$, $q=A^2$, $\{n\}:=A^{2n}-A^{-2n}$, 
$$[n]:=\{n\}/\{1\} \:,\:\: [n]!:= \prod_{j=1}^n [j] \:,\:\: [0]!:=1 \:,\:\: \qbin{n}{k}{}:= \frac{[n]!}{[k]![n-k]!}$$ 
and similarly for multinomials. 
In order to be able later on to coherently choose branches of square roots, we define
$\Delta\Hol(A)$ to be the set of functions which are holomorphic inside the unit disk $\D$ and whose square extends to a meromorphic function on $\C$.
We define $\sqrt{[k]}\in \Delta\Hol(A)$ as the root of $[k]$ which at $A=1$ is $\sqrt{k}$ and then we extend this choice on products of quantum integers so that $\sqrt{[n]!}:=\prod_{j=1}^n \sqrt{[j]}$ and $\sqrt{\frac{1}{[n]!}}:=\frac{1}{\sqrt{[n]!}}$. (Remark that these square roots exist on $\D$ because $[k]=A^{2-2k}(\sum_{i=0}^{k-1} A^{4i})$ and both terms admit a square root on $\D$).

\begin{definition*}[Unknot]
We \emph{define} the value of the colored $0$-framed unknot as follows:
\begin{equation*}
\smallunknot{1cm}\put(-10,8){$a$}:=(-1)^{2a}[2a+1].
\end{equation*}
\end{definition*}
Let also
$$\Delta(a,b,c):=\sqrt{\frac{[a+b-c]![b+c-a]![c+a-b]!}{[a+b+c+1]!}}.$$
The (unitary normalization of the) tetrahedron or $6j$-symbol is then defined as follows:
\begin{definition}[The tetrahedron or symmetric $6j$-symbol]\label{ex:tet}
Let us \emph{define} the value of the admissibly colored tetrahedron embedded as a planar graph in $\mathbb{R}^2$ as:

\begin{align}\label{eq:tet}
\Bigg(\smalltetra{1cm}\put(-13,8){$a$}\put(-8,-1){$b$}\put(-23,-1){$c$}\put(-16,-10){$d$}\put(-2,10){$f$}\put(-24,9){$e$}\ 
\Bigg)^{\mathrm{U}}:=\I^{2(a+b+c+d+e+f)}\Delta(a,b,c)\Delta(a,e,f)\Delta(d,b,f)\Delta(d,e,c) \\
\nonumber \times \sum_{k= \max T_i}^{k=\min Q_j} (-1)^{k}\begin{bmatrix}
k+1\\
k-T_1, k-T_2, k-T_3, k-T_4, Q_1-k, Q_2-k, Q_3-k,1 \end{bmatrix}
\end{align}
where $T_1=a+b+c,\ T_2=a+e+f,\ T_3=d+b+f,\ T_4=d+e+c,\  Q_1=a+b+d+e,\ Q_2=a+c+d+f, \ Q_3=b+c+e+f$. 
\end{definition}
\begin{definition}[The crossed tetrahedron]\label{ex:crossedtet} 
Let us \emph{define} the value of the admissibly colored tetrahedron, embedded in $\mathbb{R}^3$ as in the diagram below, by: 
\begin{equation*}
\smallcrossedtetra{1cm}\put(-14,17){$a$}\put(-5,-1){$b$}\put(-12,8){$c$}\put(-16,-11){$d$}\put(-22,4){$f$}\put(-33,4){$e$}: =\I^{2(e+b-a-d)}A^{2(e^2+e+b^2+b-a^2-a-d^2-d)}\ \ \Bigg( \smalltetra{1cm}\put(-13,8){$a$}\put(-8,-1){$b$}\put(-23,-1){$c$}\put(-16,-10){$d$}\put(-2,10){$f$}\put(-24,9){$e$}\ \Bigg)^{\mathrm{U}}
\end{equation*}
\end{definition}

\begin{remark}\label{rem:tet0}
If in formula \eqref{eq:tet} one color, say $f$, is $0$ then the admissibility conditions force $e=a$ and $d=b$ and, using also that $a+b+c\in \Z$ we get:
\begin{equation*}
\Bigg(\smalltetra{1cm}\put(-13,8){$a$}\put(-8,-1){$b$}\put(-23,-1){$c$}\put(-16,-10){$d$}\put(-2,10){$f$}\put(-24,9){$e$}\ \Bigg)^{\mathrm{U}}=\I^{2a+2b}\left(\sqrt{[2a+1][2b+1]}\right)^{-1}
\end{equation*}
which does in fact not depend on $c$.
\end{remark}

The preceding formulas can be used as ``building blocks'' to define invariants of colored KTG's up to isotopy in $\mathbb{R}^3$: let $(\Ga,c)$ be a colored KTG, $D\subset \mr^2$ be a diagram chosen so that the framing of $\Ga$ coincides with the blackboard framing; let $V,E$ be the sets of vertices and edges of $\Ga$, and $C,F$ the sets of crossings and edges of $D$. Since each edge of $D$ is a sub-arc of one of $\Ga$ it inherits a coloring from~$c$. Let the \emph{regions} $r_0,\ldots, r_m$ of $D$ be the connected components of $\mr^2\setminus D$ with $r_0$ the unbounded one; we will denote by $R$ the set of regions and we will say that a region ``contains'' an edge of $D$ or a crossing if its closure does. 
\begin{remark}\label{rem:unboundedregion}
Applying an isotopy to $D$ we can force any region to become the ``unbounded'' one, denoted $r_0$ above. We will exploit this freedom in the proof of Proposition \ref{prop:geometricrecursion}.
\end{remark}
\begin{definition}[Shadow-state]
A shadow-state $s$ is a map $s:R \to \frac{\mathbb{N}}{2}$ such that $s(r_0)=0$ and whenever two regions $r_i$ and $r_j$ contain an edge $f$ of $D$ then $s(r_i),s(r_j),c(f)$ form an admissible triple as in Definition \ref{def:admcol}.
\end{definition}
Given a shadow-state $s$, we can define its \emph{weight} as a product of factors indexed by the local building blocks of $D$ i.e.\ the regions of $D$, the vertices of $\Ga$, and the crossings. To define these factors explicitly, in the following we will denote by $a,b,c$ the colors of the edges of $\Ga$ (or of $D$) and by $u,v,t,w$ the values of a shadow-state on the regions (which with an abuse of terminology we will call the shadow-states of the regions).
\begin{enumerate}
\item If $r$ is a region whose shadow-state is $u$ and $\chi(r)$ is its Euler characteristic, $$w_s(r):= \Big(\smallunknot{1cm}\put(-16,-9){$u$}\Big)^{\chi(r)}=\left(\frac{(-1)^{2u}\{2u+1\}}{\{1\}}\right)^{\chi(r)}$$
\item If $v$ is a vertex of $\Ga$ colored by $a,b,c$ and $t,u,v$ are the shadow-states of the regions containing it then $$w_s(v):= \smalltetra{1cm}\put(-13,8){$a$}\put(-8,-1){$b$}\put(-23,-1){$c$}\put(-16,-10){$t$}\put(-2,10){$v$}\put(-24,9){$u$}$$
\item If $c$ is a crossing between two edges of $G$ colored by $a,b$ and $u,v,t,w$ are the shadow-states of the regions surrounding $c$ then $$w_s(c):= 
\smallcrossedtetra{1cm}
\put(-14,17){$w$}
\put(-5,-1){$t$}
\put(-12,8){$a$}
\put(-16,-11){$u$}
\put(-24,2){$b$}
\put(-33,4){$v$}$$
\end{enumerate}  
From now on, to avoid a cumbersome notation, given a shadow-state $s$ we will not explicitly write the colors of the edges of each graph providing the weight of the local building blocks of $D$ as they are completely specified by the states of the regions and the colors of the edges of $\Gamma$ surrounding the block. 
Then we may define the weight of the shadow-state $s$ as:
\begin{equation}\label{eq:shadowstate}
w(s)=\prod_{r\in R}  \smallunknot{0.5cm}^{\chi(r)}\prod_{v\in V}\nns \smalltetra{0.5cm}\prod_{c\in C}\smallcrossedtetra{0.5cm}\ .
\end{equation}
Then, since the set of shadow-states of $D$ is easily seen to be finite, we may define the $\brunit{\Ga}{c}$ as
$$\brunit{\Ga}{\c}:=\sum_{s\in {\rm shadow\ states}} w(s).$$
The following result was first proved by Reshetikhin and Kirillov \cite{KR} (see also \cite{Coint} for a skein theoretical proof) and shows that shadow state-sums provide a different approach to the computation of the so-called ``quantum spin networks'' or ``Kauffman brackets'' in their unitary normalizations $\brunit{\Ga}{c}\in \Delta\Hol(A)$:

\begin{theorem}[Shadow-state formula for Kauffman brackets.]\label{teo:shadowstate}
$\brunit{\Ga}{c}$ is a well defined invariant up to isotopy of $\Ga$.
In particular, if $\Ga$ has a planar diagram $D$ with respect to which the framing of $\Ga$ is the blackboard framing, then formula \eqref{eq:shadowstate} reduces to:
 \begin{equation}\label{eq:shadowformula}
 \brunit{\Ga}{c}=\sum_{s\in {\rm shadow\ states}}\ \ \  \prod_{r\in R}  \smallunknot{0.5cm}^{\chi(r)}\prod_{v\in V}\nns \smalltetra{0.5cm}\ \ .
\end{equation}
\end{theorem}
\begin{remark}\label{rem:squareok}
\begin{enumerate}
\item The above defined invariant coincides with the so-called ``unitary quantum spin network'' defined and studied via recoupling theory in \cite{KL}.
\item For each graph $\Ga$,  $(\brunit{\Ga}{\c})^2$ is the restriction to $\D$ of a meromorphic function on $\C$ whose poles are in the set $\{0,\exp(\frac{\I\pi p}{q}),\ p,q\in \Z\}$. Indeed a close inspection to Formula \eqref{eq:shadowstate} or \eqref{eq:shadowformula} shows that the only odd powers of square roots appearing in the expression are those associated to the colors of the edges of $\Ga$, and thus can be factored out of the whole sum in \eqref{eq:shadowformula}. (All other square roots arising from $6j$-symbols involve one edge color and two state parameters. Each such square root appears twice: once for each endpoint of the edge.)
\end{enumerate}
\end{remark}

As an example we state a simple corollary that we will use often below. Its proof is a simple application of the shadow state formula
but can also be found in \cite{MV}.

\begin{corollary}[Triangle formula]
Suppose $\Ga$ is a KTG such that three edges bound a flat triangular disk. Then we have (representing $\Gamma$ on the left hand side):
\begin{equation}\label{eq:triangleformula}
\left \langle  
\epsh{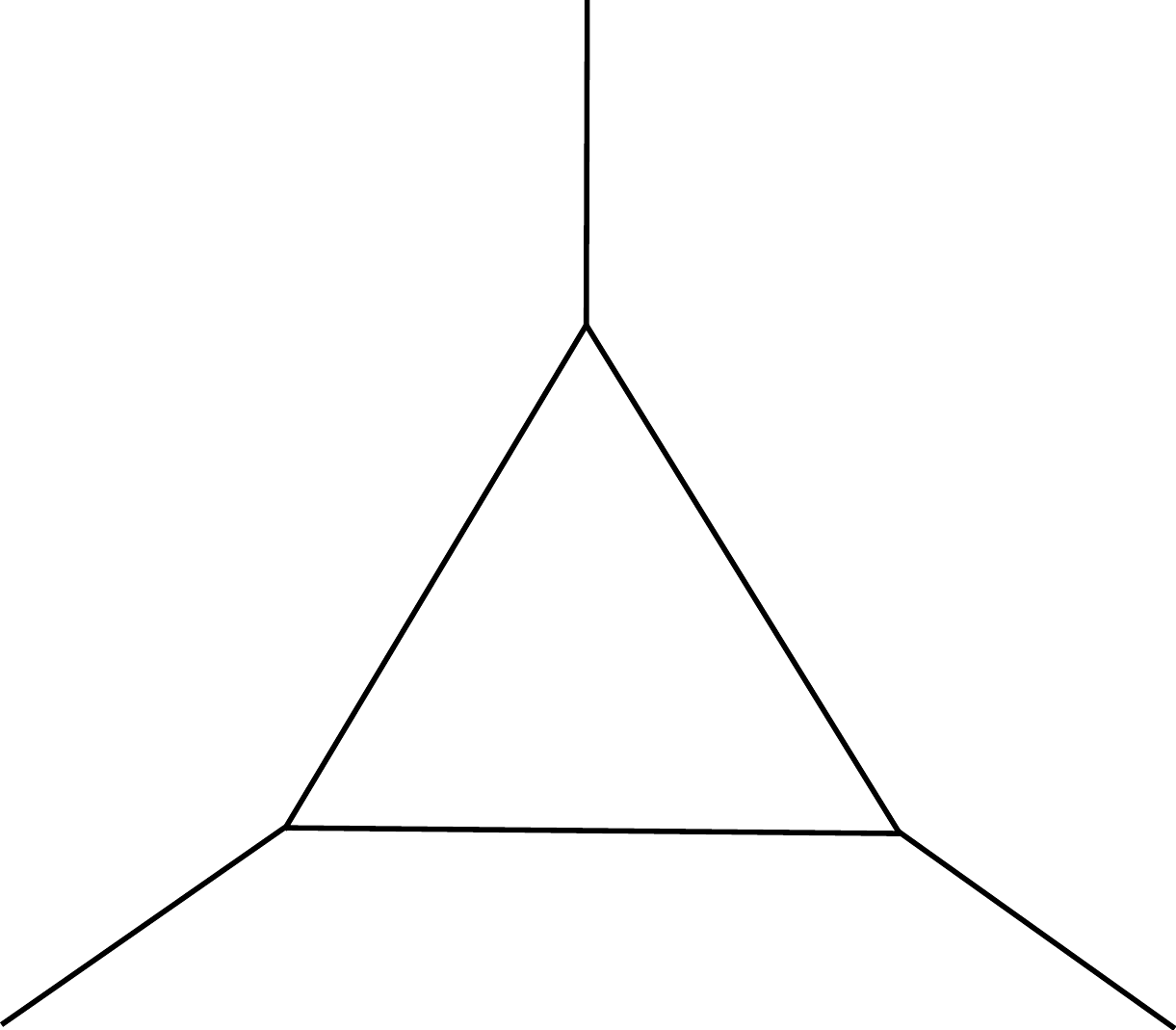}{1cm}
\put(-14,9){$a$}
\put(-34,-9){$c$}
\put(-2,-9){$b$}
\put(-10,0){$f$}
\put(-25,0){$e$}
\put(-17,-13){$d$}  
\ \right \rangle^{\mathrm{U}}
=\left \langle 
\smalltetra{1cm}
\put(-13,8){$a$}
\put(-8,-1){$b$}
\put(-23,-1){$c$}
\put(-16,-10){$d$}
\put(-2,10){$f$}
\put(-24,9){$e$} \: \:
\right \rangle^{\mathrm{U}}\:  
\left \langle 
\epsh{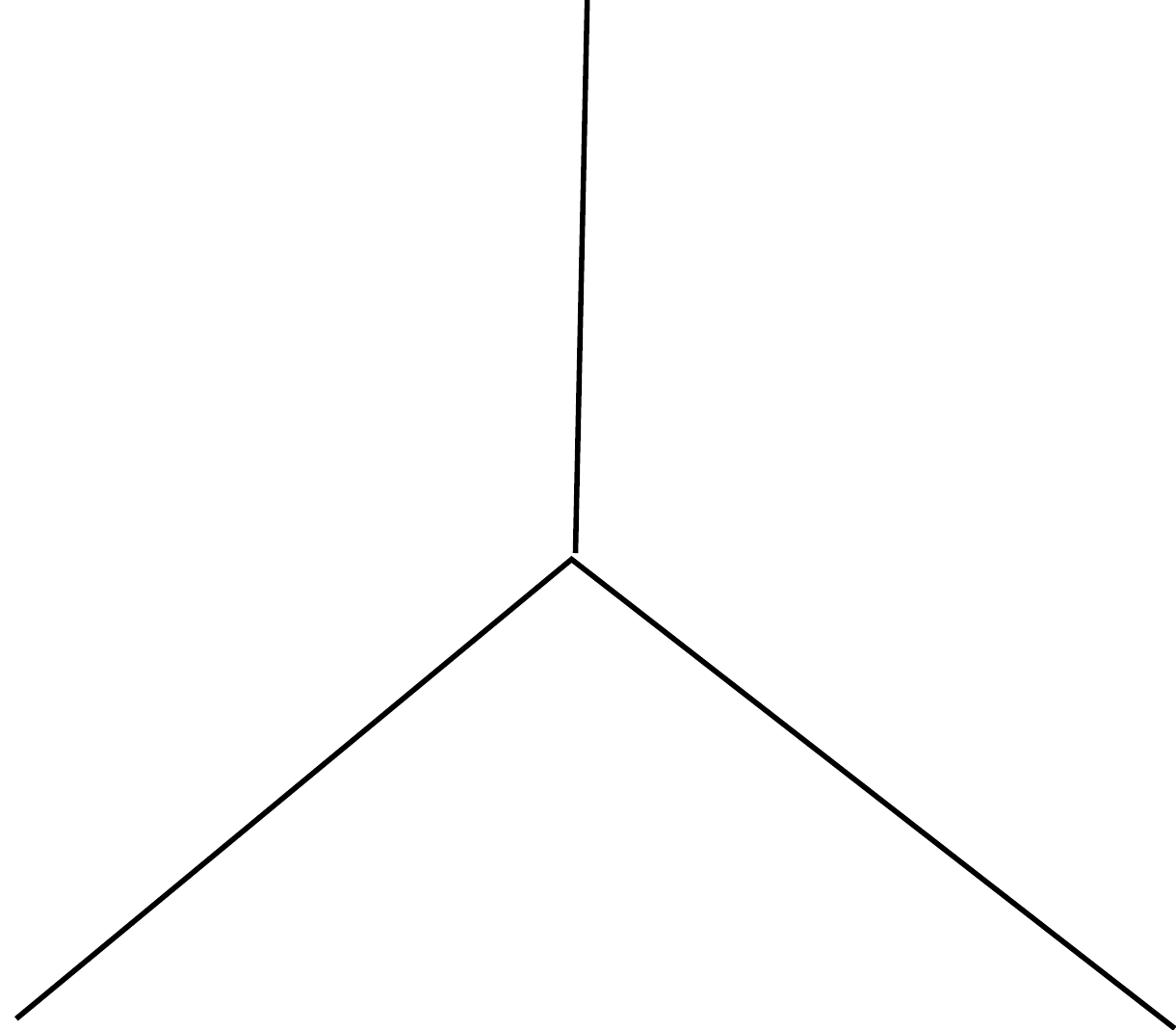}{1cm}
\put(-14,9){$a$}
\put(-30,-5){$c$}
\put(-7,-5){$b$} \:
\right \rangle^{\mathrm{U}}.
\end{equation}
\end{corollary}

\section{Recursion relations associated to faces of trivalent graphs} \label{sec:recursions}

In \cite{GaLe}, it was proved that the sequence of quantum spin networks associated to a link (better known as its \emph{colored Jones polynomials}) is \emph{holonomic}, i.e. satisfies a set of recurrence relations. More explicitly, if $k\subset \mathbb{S}^3$ is a knot, let $J_n=\brkauf{k}{\frac{n-1}{2}}$ be the sequence of its colored Jones polynomials and extend it to a map $J:\mathbb{Z}\to \mathbb{C}[A^{\pm 1}]$ by setting $J_1=1,\ J_{0}=0,\ J_{-n}=-J_{n}$. Let $S=\{f:\mz\to \mathbb{C}[A^{\pm 1}]\}$ and let $L,M:S\to S$ be defined as $L(f)(n)=f(n+1)$ and $M(f)(n)=A^{2n}f(n)$.
The following was proved in \cite{GaLe}:
\begin{theorem}
There exists a polynomial $Q_A(L,M)$ depending on $k$ such that $$Q_A(L,M)(J)=0.$$
\end{theorem}
(Actually the above statement is weaker than the original one: we refer to \cite{GaLe} for full details).
 
Even if the techniques of the above theorem may be adapted to prove similar statements for general colored KTG invariants, for example
along the lines of \cite{DGV}, it is difficult to make this explicit in general.
In this section, we shall consider the case when $\Ga$ is the $1$-skeleton of a polyhedron $P$ and, given the embedded cycle $\partial D\subset \Ga$ formed by the boundary of a face of $P$, we will provide a family of instances of these recursion polynomials. (Note that, as there are now many ``colors'' given by the map $c:E\rightarrow \frac{\mathbb{N}}{2}$, we should also expect many recursions. Our technique can be easily generalised to any knotted trivalent graph $\Ga$ and any cycle in $\Gamma$ which is the boundary of a disc embedded in $S^3\setminus \Ga$, but we won't need this level of generality here.)

Let $\Ga$ be a KTG which is the 1-skeleton of a polyhedron $\P$. Let $\c:\E\to \frac{\mathbb{N}}{2}$ be an admissible coloring on the edges of $\Ga$. Let $D\subset \mathbb S^3\setminus \Ga$ be an oriented disc isotopic to a face of $\P$.  Fix an orientation of $D$, a basepoint $p_0\in \partial D$ and denote by $e_0,\ldots e_{p-1},e_p=e_0$ the edges contained in $\partial D$ (numbered by walking along $\partial D$ in the clockwise direction starting from $p_0$). Let $a_i,b_i,  i=0\ldots p-1$ be the colors respectively of $e_i$ and of the edge of $\Ga$ sharing a vertex with $e_i,e_{i+1} (\mathrm{mod} \; p)$ (see Figure \ref{fig:notation}). 
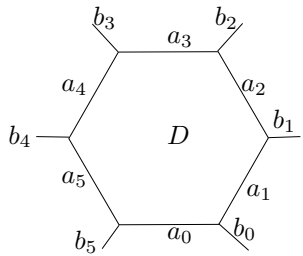
\begin{figure}
$\epsh{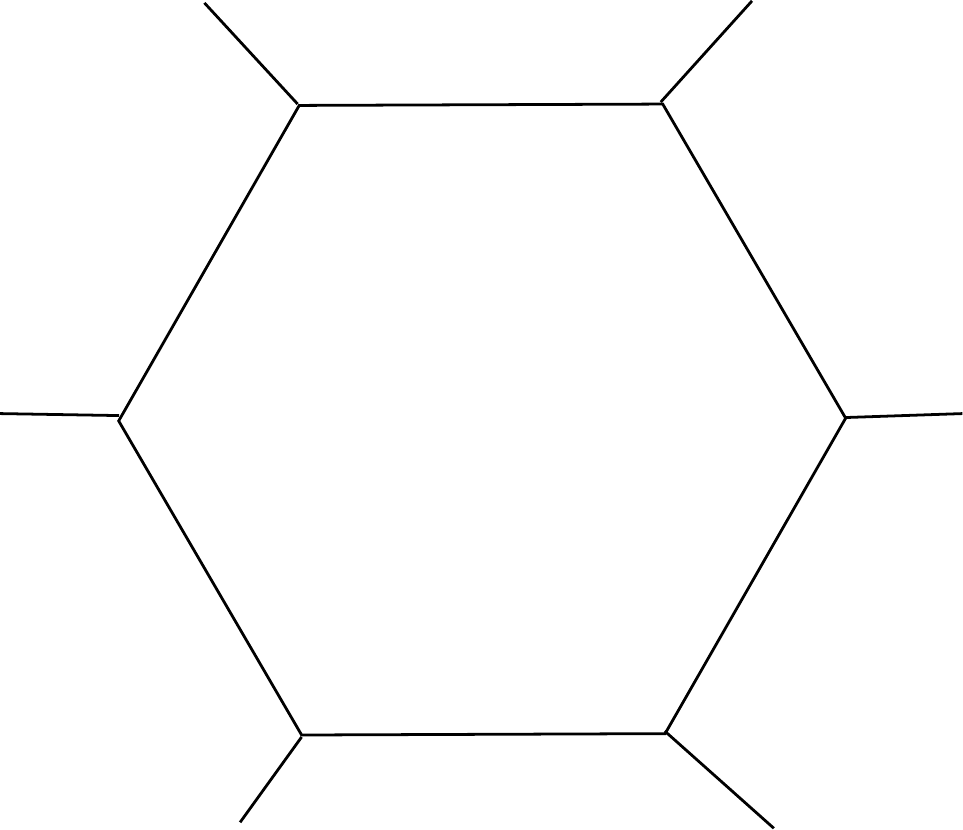}{20ex}
\put(-50,-36){$a_0$}
\put(-20,-20){$a_1$}
\put(-25,-35){$b_0$}
\put(-22,20){$a_2$}
\put(-10,06){$b_1$}
\put(-32,45){$b_2$}
\put(-78,45){$b_3$}
\put(-110,00){$b_4$}
\put(-85,-40){$b_5$}
\put(-50,38){$a_3$}
\put(-90,-15){$a_5$}
\put(-90,20){$a_4$}
\put(-50,0){$D$}$
\caption{Some of the notation for Proposition \ref{prop:geometricrecursion}.}\label{fig:notation}
\end{figure}

\begin{definition}[Admissible perturbations]\label{def:admperturb}
An \emph{admissible perturbation} of the coloring $c$ of $\Gamma$ around $D$ is a map $\delta:\{e_0,\ldots e_{p-1}\}\to \{\pm \frac{1}{2}\}$ such that the coloring $c+\delta$ (defined to be equal to $c$ on the edges outside $\partial D$ and to be $e_i\mapsto a_i+\delta(e_i),\ i=0,\ldots p-1$ on $\partial D$) is admissible.
\end{definition}
We consider the following proposition as a ``quantum version'' of  Proposition \ref{prop:circlegeometric}:
\begin{proposition}[Circle recursions]\label{prop:geometricrecursion}
For any pair $(\delta(e_0),\delta(e_{p-1}))\in \{\pm\frac{1}{2}\}^2$, the following holds:
\begin{equation} \label{eq:recursion0} 
\sum_{\delta} k(\c,\delta) \brunit{\Ga}{c+\delta} = Z(\c,\delta(e_0),\delta(e_{p-1}))\brunit{ \Ga}{c}
\end{equation} 
where $\delta$ ranges over all admissible perturbations of $c$ around $D$ whose value on $e_0$ and $e_{p-1}$ are respectively $\delta(e_0), \delta(e_{p-1})$, and the functions $Z(c,\delta(e_0),\delta(e_{p-1}))$ and $k(c,\delta) $ are defined as follows (here to keep the notation light we shall write $\delta_i$ for $\delta(e_i)$ and 
$a'_i$ for $a_i+\delta_i$). Define $Z(c,\delta(e_0),\delta(e_{p-1})) = \QM(p-1)$ and $k(c,\delta)= \prod_{i=0}^{p-2} \QM(i)$ in terms of the quotient $\QM$:
\begin{equation}\label{eq:QM}
\QM(j):=\frac{\smalltetra{1.2cm}\put(-27,8){$a_j$}\put(-31,-1){$b_j$}\put(-12,-1){$a_{j+1}$}\put(-26,-10){$a'_{j+1}$}\put(-2,10){$\frac{1}{2}$}\put(-43,9){$a'_j$}}{\smalltetra{1.2cm}\put(-27,8){$a'_j$}\put(-31,-1){$b_j$}\put(-12,-1){$a'_{j+1}$}\put(-26,-10){$a'_{j+1}$}\put(-2,10){$0$}\put(-43,9){$a'_j$}}  \ .
\end{equation}
\end{proposition}

\begin{proof}
Let $\Ga'=\Ga'(\delta_0,\delta_{p-1})$ be the colored planar graph obtained by modifying $\Ga$ around $e_0\cap e_{p-1}$ as follows:
\begin{equation}
\epsh{aftertrianglecontraction.pdf}{1.5cm}
\put(-23,12){$b_{p-1}$}
\put(-46,-8){$a_0$}
\put(-12,-8){$a_{p-1}$} 
\put(-29,-15){$D$}  
\ \ \longrightarrow \ \  
\epsh{beforetrianglecontraction.pdf}{1.5cm}
\put(-23,14){$b_{p-1}$}
\put(-51,-14){$a_0$}
\put(-4,-15){$a_{p-1}$}
\put(-15,0){$a'_{p-1}$}
\put(-42,0){$a'_0$}
\put(-29,-17){\Small{$1/2$}}  
\end{equation}
where in the picture both the framings of $\Ga$ and $\Ga'$ are supposed to be the blackboard framings.  We will call $r_0$ the region bounded by the small triangle created by the move, and let $r'_0$ be its complement in $D$. 

The idea of the proof is to express $\langle\Ga'\rangle^{\mathrm{U}}$ in terms of $\brunit{\Ga}{c+\delta}$ two ways.
First directly using the triangle formula 
\eqref{eq:triangleformula}. Second by applying the shadow state formula \eqref{eq:shadowformula} with region $r_0$ having the role of the ``unbounded" region
and comparing it to the similar state sum for various colorings of $\langle\Ga\rangle$.

If we multiply \eqref{eq:recursion0} by the denominator $Z'$ of $Z$ then
the right hand side becomes
\[
\smalltetra{1.2cm}\put(-27,8){$a'_{0}$}\put(-33,-1){$a'_{p-1}$}\put(-12,-1){$b_{p-1}$}\put(-26,-11){$a_{p-1}$}\put(-2,10){$a_{0}$}\put(-40,9){$\frac{1}{2}$}\hspace{10pt} \brunit{\Ga}{c} = \langle\Ga'\rangle^{\mathrm{U}},
\]
using \eqref{eq:triangleformula}. 
The proof is finished once we show that the left hand side of the same equation equals the expression of
$\langle\Ga'\rangle^{\mathrm{U}}$ in terms of the shadow state formula with distinguished region $r_0$. In other words we compute $\langle \Ga'\rangle^{\mathrm{U}}$ by stipulating the triangular region $r_0$ to be the one carrying shadow state $0$. Let $R_{\Gamma'}$ be the set of all regions of $\Gamma'$. 
The region $r'_0$ is a disc and its only admissible shadow state is $\frac{1}{2}$. Moreover in a neighborhood of $r'_0$ we will see one region $r^*_i\in R_{\Gamma'}$ for each edge $e_i$ ($\partial r^*_i$ contains $e_i$, and $r^*_i\cap r_0=\emptyset$ unless $i=p-1$ or $i=0$, in which case the intersection is one of the edges of $r_0$, colored respectively by $a'_{p-1}$ and by $a'_0$). 

We have $\langle\Ga'\rangle^{\mathrm{U}}=\sum_{\text{shadow states}} w'(s)$  where $w'(s)$ is the weight of the shadow state $s:R_{\Gamma'}\rightarrow \frac{\mathbb{N}}{2}$ for $\Ga'$. Observe that the only admissible shadow-states of $r^*_i$ are $a'_i=a_i+\delta_i$ with $\delta$ admissible (as in the statement). So we may also write:
$$\langle\Ga'\rangle^{\mathrm{U}}=\sum_{\delta}\sum_{s\in S(\delta)} w'(s)=\sum_{\delta} w'(\delta)$$
where for each perturbation $\delta$ we let $S(\delta)$ be the set of all the shadow-states whose values on regions $r^*_i$ are $a'_i=a_i+\delta_i$ for all $i\in \{0,\ldots p-1\}$ and  $w'(\delta):=\sum_{s\in S(\delta)} w'(s)$. We claim that $w'(\delta)= \brunit{\Ga}{\c+\delta} k(c,\delta) Z'$ for all admissible~$\delta$, which (by summing over $\delta$) will finish the proof, since $Z'$ is independent of $\delta$.

Indeed, fix $\delta$. By the shadow-state formula \eqref{eq:shadowformula} 
applied to $\Ga$ with the region $D$ colored by $0$, one sees that $\brunit{\Ga}{c+\delta}= \sum_{s\in S(\delta)} w(s)$ is a state-sum over the \emph{same} set $S(\delta)$ of colorings (of $R_\Gamma \smallsetminus \{D\} \simeq R_{\Gamma'} \smallsetminus \{r_0, r'_0\}$) as $\langle \Gamma' \rangle ^{\mathrm{U}}$ is  --- the colorings being extended to be $0$ on $D$ instead of $\frac{1}{2}$ on $r'_0$ and $0$ on $r_0$. Thus the only difference is in the weights associated to $r'_0$ and to the vertices contained in $\partial D$. 
More explicitly we claim that for any $s\in S(\delta)$ one has 
$\frac{w'(s)}{w(s)}=k(c,\delta) Z'$ (independent of $s$).
Indeed the weights differ only by the vertices touching $D$ and by the presence of the weight of the region $r'_0$ (which is $-[2]$):
$$w'(s)=-[2]\ \ \ 
\smalltetra{1.2cm}
\put(-27,8){$\frac{1}{2}$}
\put(-33,-3){$a'_{0}$}
\put(-12,-1){$a_{0}$}
\put(-22,-13){$a'_{0}$}
\put(-2,10){$\frac{1}{2}$}
\put(-38,9){$0$}\ \ \ \ \ \ 
\smalltetra{1.2cm}
\put(-26,8){$\frac{1}{2}$}
\put(-39,-3){$a'_{p-1}$}
\put(-12,-1){$a_{p-1}$}
\put(-26,-13){$a'_{p-1}$}
\put(-2,10){$\frac{1}{2}$}
\put(-39,9){$0$} \ \ \ \ \ \ \ 
\smalltetra{1.2cm}
\put(-24,8){$0$}
\put(-39,-3){$a'_{p-1}$}
\put(-12,-1){$a'_{p-1}$}
\put(-26,-13){$b_{p-1}$}
\put(-2,10){$a'_0$}
\put(-43,9){$a'_0$} \ \ \ \ 
\prod_{i=0}^{p-2}\  \  
\smalltetra{1.2cm}
\put(-27,8){$a_i$}
\put(-31,-1){$b_i$}
\put(-12,-1){$a_{i+1}$}
\put(-26,-10){$a'_{i+1}$}
\put(-2,10){$\frac{1}{2}$}
\put(-43,9){$a'_i$} 
\hspace{4pt} \times [\text{off-$D$ terms}]$$
$$w(s)= \ \ 
\smalltetra{1.2cm}
\put(-23,8){$0$}
\put(-39,-3){$a'_{p-1}$}
\put(-12,-1){$a'_{p-1}$}
\put(-26,-13){$b_{p-1}$}
\put(-2,10){$a'_0$}
\put(-43,9){$a'_0$} \ \ \ \ \ 
\prod_{i=0}^{p-2} \ \  \ 
\smalltetra{1.2cm}
\put(-27,8){$a'_i$}
\put(-31,-1){$b_i$}
\put(-12,-1){$a'_{i+1}$}
\put(-26,-10){$a'_{i+1}$}
\put(-2,10){$0$}
\put(-43,9){$a'_i$}\ \ \ 
\times [\text{off-$D$ terms}].$$  
Taking the ratio (all the symbols represent nonzero functions in $\Delta\Hol(A)$) and using Remark \ref{rem:tet0}, we get:
\begin{eqnarray*}
\frac{w'(s)}{w(s)}&=& -[2]\frac{
\smalltetra{1.2cm}
\put(-23,8){$\frac{1}{2}$}
\put(-32,-3){$a'_{0}$}
\put(-12,-1){$a_{0}$}
\put(-23,-13){$a'_{0}$}
\put(-2,10){$\frac{1}{2}$}
\put(-38,9){$0$}\ \ \ \ \ 
\smalltetra{1.2cm}
\put(-23,8){$\frac{1}{2}$}
\put(-39,-3){$a'_{p-1}$}
\put(-12,-1){$a_{p-1}$}
\put(-26,-13){$a'_{p-1}$}
\put(-2,10){$\frac{1}{2}$}
\put(-38,9){$0$}\ \ \ \ \ \  
\smalltetra{1.2cm}
\put(-23,8){$0$}
\put(-39,-3){$a'_{p-1}$}
\put(-12,-1){$a'_{p-1}$}
\put(-26,-13){$b_{p-1}$}
\put(-2,10){$a'_0$}
\put(-43,9){$a'_0$} }
{\smalltetra{1.2cm}
\put(-23,8){$0$}
\put(-39,-3){$a'_{p-1}$}
\put(-12,-1){$a'_{p-1}$}
\put(-26,-13){$b_{p-1}$}
\put(-2,10){$a'_0$}
\put(-43,9){$a'_0$} }\ \ \ 
\prod_{i=0}^{p-2} \ \  \ \frac{ 
\smalltetra{1.2cm}
\put(-27,8){$a_i$}
\put(-31,-1){$b_i$}
\put(-12,-1){$a_{i+1}$}
\put(-26,-10){$a'_{i+1}$}
\put(-2,10){$\frac{1}{2}$}
\put(-43,9){$a'_i$}}
{\smalltetra{1.2cm}
\put(-27,8){$a'_i$}
\put(-31,-1){$b_i$}
\put(-12,-1){$a'_{i+1}$}
\put(-26,-10){$a'_{i+1}$}
\put(-2,10){$0$}
\put(-43,9){$a'_i$}}
\\ &=& Z' k(\c,\delta)
\end{eqnarray*}
by definition of $Z'$ and $k(\c,\delta)$. Multiplying by $w(s)$ and summing over $s\in S(\delta)$ gives the desired identity.
\end{proof}

To illustrate our recursions consider the simplest case of a triangular face $p=3$.
We recover the well-known three-term recursion for the $6j$-symbol named after Gordon and Schulten \cite{SG}.
As the proof shows, the Gordon-Schulten recursion factorizes into two instances of the 
circle recursion, suggesting that the circle recursion is more fundamental.

\begin{corollary}[Gordon-Schulten recursion]
Let $e_0,e_1,e_2$ be a triangular face of $\Ga$ and define the edge coloring $1_i$ by $1_i(e_j) = \delta_{i,j}$.
Also define $\delta_\pm=\frac{1}{2}(1_0\pm1_1+1_2)$.
\[
\sum_\pm\frac{k(c,\delta_\pm)k(c+\delta_\pm,-\delta_\mp)}{Z(c+\delta_\pm,\frac{-1}{2},\frac{-1}{2})}\brunit{\Ga}{c\pm 1_1} =
\]
\[ 
\left(-Z(c,\frac{1}{2},\frac{1}{2})+\sum_\pm \frac{k(c,\delta_\pm)k(c+\delta_\pm,-\delta_\pm)}{Z(c+\delta_\pm,\frac{-1}{2},\frac{-1}{2})} \right)\brunit{\Ga}{c}
\]
\end{corollary}
\begin{proof}
For $\epsilon \in \{-1,1\}$ the only two admissible perturbations $\delta$ with $\delta(e_0)= \delta(e_2)=\frac{\epsilon}{2}$ 
are $\epsilon\delta_\pm$. Hence the circle recursion reads:
\[
\sum_{\pm}k(c,\epsilon\delta_\pm)\brunit{\Ga}{c+\epsilon\delta_\pm} = Z(c,\frac{\epsilon}{2},\frac{\epsilon}{2})\brunit{\Ga}{c}.
\]
Applying the recursion for $\epsilon = -1$ to both terms on the left hand side of the recursion for $\epsilon=1$
yields the result.
\end{proof}

The coefficients of the recursion may be written out explicitly using the formulae in the next section
but as we make no further use of this example we have not printed them here. See for example \cite{TaWo},
Prop. 4.1.1.

\section{Proof of Theorem \ref{teo:main}} \label{sec:mainproof}
Let us fix the notation for all this section. 
Let $\Ga$ be the $1$-skeleton of a hyperbolic, hyperideal polyhedron $\P$ and let $\gamma_i:E(\P)\to (0,\pi]$
be the exterior dihedral angles of $\P$. 
Let also $\lambda\subset \Ga$ be an embedded oriented curve forming the boundary of a face $D$ of $\P$ and $\lambda^{\mathrm{trunc}}$ be the associated curve in $\Ptr$,  $e_0,e_1,\ldots e_{p-1}$ be the sequence of edges of $\Ga$ contained in $\lambda$ and $v_0,v_1,\ldots v_{2p-2},v_{2p-1}$ be the vertices of $\Ptr$ contained in $\lambda^{\mathrm{trunc}}$ so that $\overrightarrow{v_{2i}v_{2i+1}}\subset e_i$. 
 
We will consider a sequence $\c^{(n)}$ of $\mn$-valued colorings on $\Ga$ such that for each edge, $\lim_{n\to \infty} \frac{\c^{(n)}(e_i)}{n}=1- \frac{\gamma_i}{2\pi}$ (and so in particular $\c^{(n)}(e_i)\in [\frac{n}{2},n)$ definitely).
Finally we will let $a_i:=\c^{(n)}(e_i)$ (suppressing the $n$ from the notation with a little abuse) and let $b_i$ be the color of the edge of $\Ga$ sharing a vertex with $e_i$ and $e_{i+1}$ with respect to the coloring $\c^{(n)}$. Our goal is to relate the recurrence equations provided by Proposition \ref{prop:geometricrecursion} to the geometric equations associated by Proposition \ref{prop:circlegeometric} to $\lambda^{\rm{trunc}}$.

\subsection{Asymptotical behavior of recursion coefficients}\label{sub:asympanalysis}
Let us compute explicitly the values of the coefficients $k(c,\delta)$ of Proposition \ref{prop:geometricrecursion}. 
Using formula \eqref{eq:tet}, we get that $\QM(j)$ from \eqref{eq:QM} is a matrix whose rows and columns are indexed respectively by $\delta(e_j),\delta(e_{j+1})\in \{\pm \frac{1}{2}\}$, namely: 
\begin{equation*}
\renewcommand\arraystretch{2}
\begin{array}{c|c|c} 
\QM(j)= & {\rm if }\delta_j=\frac{1}{2} & {\rm if } \delta_j=-\frac{1}{2}\\
\hline
{\rm if } \delta_{j+1}=\frac{1}{2} & \I^{-1}\sqrt{\frac{[1+a_j + a_{j+1} - b_j] [2+a_j + a_{j+1} + b_j]}{
 [2 + 2 a_j]  [2 + 2 a_{j+1}] }} &\I\sqrt{\frac{[1-a_j + a_{j+1} + b_j][-a_{j+1}+a_j+b_j]}{[ 2 a_j][2a_{j+1}+2]}}\\
\hline
{\rm if } \delta_{j+1}=-\frac{1}{2} & \I\sqrt{\frac{[1+a_j - a_{j+1} + b_j][a_{j+1}-a_j+b_j]}{[2a_j+2][2 a_{j+1}]}} &  \I\sqrt{\frac{[a_j + a_{j+1}- b_j][a_j+a_{j+1}+b_j+1]}{[2a_j][2a_{j+1}]}}
\end{array}\end{equation*}
where, as explained in Section \ref{sec:shadowstates}, we chose the square roots by stipulating that for all $k\in \mn$, $\sqrt{[k]}$ was $\sqrt{k}$ at $A=1$, and so by holomorphicity $\sqrt{[k]}$ is a positive real number on the arc $A=\exp(\I \theta)$ for $\theta\in [0,\frac{\pi}{k})$ and it is a positive real multiple of $-\I$ on $\theta \in (\frac{\pi}{k}, \frac{2\pi}{k})$. Equivalently, if $A=\exp(\frac{\I\pi}{2n})$ then $\sqrt{[k]}\in \mr_+$ if $k< n$ and it is a positive real multiple of $-\I$ if $k\in (n,2n)$.

Let us now compute the asymptotical behavior of the recursion coefficients found in 
the table above
when we evaluate at $A=\exp(\frac{\I\pi}{2n})$. 
By hypothesis there exist angles $\gamma_j, \gamma_{j+1}, \beta_j\in [0,\pi)$ such that 
$$\lim_{n\to \infty} 2\pi (1-\frac{a_j}{n})=\gamma_j,\ \lim_{n\to \infty} 2\pi (1-\frac{a_{j+1}}{n})=\gamma_{j+1},\ {\rm and}\  \lim_{n\to \infty} 2\pi (1-\frac{b_j}{n})=\beta_j$$ (and so in particular $a_j,b_j,a_{j+1}\in [\frac{n}{2},n)$). Remark that since these angles satisfy Bonahon and Bao's conditions, one has $\beta_j+\gamma_j+\gamma_{j+1}>2\pi$, and replacing this to compute the square roots of the evaluation of the above quantum integers we get:
$$\ev_n\left(\sqrt{[2a_j]}\right)\sim\ev_n\left(\sqrt{[2+2a_j]}\right)\sim -\I \sqrt{\frac{\sin(\gamma_j)}{\sin(\frac{\pi}{n})}},$$
$$ \ev_n\left(\sqrt{[2a_{j+1}]}\right)\sim \ev_n\left(\sqrt{[2+2a_{j+1}]}\right)\sim -\I \sqrt{\frac{\sin(\gamma_{j+1})}{\sin(\frac{\pi}{n})}}$$
$$\ev_n\left(\sqrt{[a_j-a_{j+1}+b_j]}\right)\sim \ev_n\left(\sqrt{[1+a_j-a_{j+1}+b_j]}\right)\sim \sqrt{\frac{\sin(\frac{\gamma_j}{2}-\frac{\gamma_{j+1}}{2} +\frac{\beta_j}{2})}{\sin(\frac{\pi}{n})}}$$

$$\ev_n\left(\sqrt{[1+a_j+a_{j+1}+b_j]}\right)\sim\ev_n\left(\sqrt{[2+a_j+a_{j+1}+b_j]}\right)\sim -\I \sqrt{\frac{|\sin(\frac{\gamma_j}{2} +\frac{\gamma_{j+1}}{2} +\frac{\beta_j}{2})|}{\sin(\frac{\pi}{n})}},$$
where the symbol $f_n\sim g_n$ stands for $\lim_{n\to \infty} \frac{f_n}{g_n}=1$.
Then the matrix $\QM(j)$ is asymptotically equivalent to:
$$
\frac{\begin{pmatrix}
   \sqrt{|\sin \frac{\gamma_j+\gamma_{j+1}-\beta_j}{2} \sin \frac{\gamma_j+\gamma_{j+1}+\beta_j}{2} |}&
-\I\sqrt{ \sin \frac{\gamma_j-\gamma_{j+1}+\beta_j}{2} \sin \frac{\gamma_{j+1}-\gamma_j+\beta_j}{2}  } \\
-\I\sqrt{ \sin \frac{\gamma_j-\gamma_{j+1}+\beta_j}{2} \sin \frac{\gamma_{j+1}-\gamma_j+\beta_j}{2}  } & 
  -\sqrt{|\sin \frac{\gamma_j+\gamma_{j+1}-\beta_j}{2} \sin \frac{\gamma_j+\gamma_{j+1}+\beta_j}{2} |}
\end{pmatrix}}{\sqrt{\sin(\gamma_j)\sin(\gamma_{j+1})}}.
$$
A straightforward computation shows that
\begin{equation}\label{eq:asympK}
\ev_n\left(\QM(j) \right)\sim R\frac{S(j)}{\sqrt{\sin(\gamma_j)\sin(\gamma_{j+1})}}R^{-1}U 
\end{equation}
 where 
\begin{equation*}
S(j)=\begin{pmatrix}
\sqrt{\sin(\frac{\gamma_j-\gamma_{j+1}+\beta_j}{2})\sin(\frac{\gamma_{j+1}-\gamma_j+\beta_j}{2})} & \sqrt{-\sin{(\frac{\gamma_j+\gamma_{j+1}-\beta_j}{2})}\sin(\frac{\gamma_j+\gamma_{j+1}+\beta_j}{2})}\\
\sqrt{-\sin{(\frac{\gamma_j+\gamma_{j+1}-\beta_j}{2})}\sin(\frac{\gamma_j +\gamma_{j+1}+\beta_j}{2})} & \sqrt{\sin(\frac{\gamma_j-\gamma_{j+1}+\beta_j}{2})\sin(\frac{\gamma_{j+1}-\gamma_j+\beta_j}{2})}
\end{pmatrix} 
\end{equation*}
and 
\begin{equation*}
R:= \begin{pmatrix}
\exp(\frac{\I\pi}{4}) & 0\\
0 & \exp(-\frac{\I\pi}{4})
\end{pmatrix}, \: 
U:= \begin{pmatrix} 0 & -\I\\ -\I & 0 \end{pmatrix}.
\end{equation*}

\begin{lemma}\label{lem:matrices}
The following holds:
$$\frac{S(j)}{\sqrt{\sin(\gamma_j)\sin(\gamma_{j+1})}}=C(\ell_{2j+1,2j+2})=\begin{pmatrix}
\cosh (\frac{\ell_{2j+1,2j+2}}{2}) & \sinh (\frac{\ell_{2j+1,2j+2}}{2}) \\
\sinh (\frac{\ell_{2j+1,2j+2}}{2}) & \cosh (\frac{\ell_{2j+1,2j+2}}{2}) 
\end{pmatrix} $$
where $\ell_{2j+1,2j+2}$ is the length of the edge opposite to the angle $\beta_j$ in a hyperbolic triangle whose exterior dihedral angles are $\gamma_{j},\gamma_{j+1},\beta_j$, or, equivalently, is the length of the edge $\overrightarrow{v_{2j+1}v_{2j+2}}$ in $\Ptr$.
\end{lemma}
\begin{proof}
Remark that, by the inequalities examined in Example \ref{ex:tetrahedron}, all the arguments of the square roots in the coefficients of $S_j$ are positive. Moreover a direct computation shows $\det(S_j)=\sin(\gamma_j)\sin(\gamma_{j+1})>0$.
So $M:=\frac{S_j}{\sqrt{\sin(\gamma_j)\sin(\gamma_{j+1})}}\in \mathrm{SL}_2(\mathbb{R})$ is a real positive matrix fixing $\pm 1\in \partial_\infty \H^2$; hence it represents a hyperbolic translation of a certain length $\lambda\in \mathbb{R}^+$, and has the form $M=C(\lambda)$. Identifying diagonal terms of $C(\lambda)$ yields $${\textstyle \cosh \frac{\lambda}{2}=\sqrt{\frac{\sin \frac{\gamma_j-\gamma_{j+1}+\beta_j}{2} \sin \frac{\gamma_{j+1}-\gamma_j+\beta_j}{2}}{\sin \gamma_j \sin \gamma_{j+1}}}}~\text{ hence }
\cosh \lambda = \frac{\cos \gamma_j \cos \gamma_{j+1} -\cos \beta_j }{\sin \gamma_j \sin \gamma_{j+1}}.$$
This equals $\cosh(\ell_{2j+1,2i+2})$ by Proposition \ref{prop:alkashi}, so $\lambda=\ell_{2j+1,2i+2}$. 
\end{proof}
Putting Lemma \ref{lem:matrices} and Equation \eqref{eq:asympK} together we get:
\begin{proposition}\label{prop:kasymp}
The following limit holds:
$$\lim_{n\to \infty} \ev_n(\QM(j))=RC(\ell_{2j+1,2j+2})R^{-1}U.
$$
Hence, letting $k(\c_n,\delta)$ be as defined in Proposition \ref{prop:geometricrecursion} and denoting by $X_{\delta,\delta'}$ the entry at row $\delta$ and column $\delta'$ of any given matrix $X$, the following holds:
$$\lim_{n\to \infty} \ev_n(k(\c_n,\delta))=\prod_{j=0}^{p-1}\left(RC(\ell_{2j+1,2j+2})R^{-1}U\right)_{\delta(e_j),\delta(e_{j+1})}.
 \qed $$
\end{proposition}

Similarly, evaluating the $Z(\c_n,\delta(e_0),\delta(e_{p-1}))$ 
of Proposition \ref{prop:geometricrecursion}, we get:
\begin{align*}
\renewcommand\arraystretch{2}
\begin{array}{c|c|c} 
Z= & {\rm if }\delta(e_0)=\frac{1}{2} & {\rm if } \delta(e_0)=-\frac{1}{2}\\
\hline
{\rm if } \delta_{e_{p-1}}=\frac{1}{2} & \I^{-1}\sqrt{\frac{[1+a_0 + a_{p-1} - b_{p-1}] [2+a_0 + a_{p-1} + b_0]}{
 [2 + 2 a_0]  [2 + 2 a_{p-1}] }} &\I\sqrt{\frac{[1-a_{p-1} + a_{0} + b_{p-1}][-a_{0}+a_{p-1}+b_{p-1}]}{[ 2 a_0][2a_{p-1}+2]}}\\
\hline
{\rm if } \delta_{e_{p-1}}=-\frac{1}{2} & \I\sqrt{\frac{[1+a_{p-1} - a_{0} + b_{p-1}][a_{0}-a_{p-1}+b_j]}{[2a_0+2][2 a_{p-1}]}} &  \I\sqrt{\frac{[a_0 + a_{p-1}- b_{p-1}][a_0+a_{p-1}+b_0+1]}{[2a_{p-1}][2a_{0}]}}
\end{array}
\end{align*}
and arguing exactly as in the preceding case we get the following:
\begin{lemma} \label{lem:Sasymp}
We have the following asymptotic behavior:
\begin{align*}\lim_{n\to \infty} \ev_n (Z(\c_n,\delta(e_0),\delta(e_{p-1})))=R
\begin{pmatrix}
\cosh (\frac{\ell_{p-1,0}}{2}) & \sinh (\frac{\ell_{p-1,0}}{2}) \\
\sinh (\frac{\ell_{p-1,0}}{2}) & \cosh (\frac{\ell_{p-1,0}}{2}) 
\end{pmatrix} R^{-1}U
\end{align*}
where $\ell_{p-1,0}$ is the length of the edge opposite to the angle $\beta_{p-1}$ in the hyperbolic triangle whose exterior dihedral angles are $\gamma_0,\gamma_{p-1},\beta_{p-1}$.
\end{lemma}

\subsection{Proof of Theorem \ref{teo:main}}
To each edge of $\lambda^{\mathrm{trunc}}$ we associate its length $\ell_{j,j+1}:=\ell(\overrightarrow{v_{j}v_{j+1}}),\ \forall j\in \{0,1,\ldots 2p-1\}$ and, as in Section \ref{sec:faceequations}, let :
\begin{align} \notag
C(\overrightarrow{v_jv_{j+1}}):=\begin{pmatrix}
\cosh\frac{\ell_{j,j+1}}{2} & \sinh\frac{\ell_{j,j+1}}{2} \\
\sinh\frac{\ell_{j,j+1}}{2} & \cosh\frac{\ell_{j,j+1}}{2}
\end{pmatrix},\:
M:=\frac{1}{\sqrt{2}}\begin{pmatrix}
1 & -1 \\ 1 & 1
\end{pmatrix}.
\end{align}

Furthermore, suppose that Assumption $1$ holds for the sequence $\brunit{\Ga}{\c_n}$ and let $f:A(\P)\to \mathbb{C}$ be the function of class $C^1$ whose existence is postulated by the assumption.
Then for each  $j\in \{0,\ldots, p-1\}$ we associate to the segment $\overrightarrow{v_{2j}v_{2j+1}}$ the ``asymptotic quantum length'' $2\frac{\partial f}{\partial \gamma_j}$ and a diagonal matrix as follows:
$$Q(\overrightarrow{v_{2j}v_{2j+1}})=\begin{pmatrix}
\frac{\partial f}{\partial \gamma_j} & 0\\
0 & -\frac{\partial f}{\partial \gamma_j}
\end{pmatrix}.$$
\begin{theorem}\label{teo:mainteoproof}
The function $f$ satisfies the following matrix-valued differential equation:
$$
\prod_{j=0}^{p-1} MC(\overrightarrow{v_{2j+1}v_{2j+2}})M Q(\overrightarrow{v_{2j}v_{2j+1}}) =-\mathrm{Id} 
$$ Moreover if $V(\gamma):A(\P)\to \mr$ is the hyperbolic volume of $\Ptr$ then $V$ also satisfies the above equations.
\end{theorem}
\begin{proof}
Fix values $\delta(e_0),\delta(e_{p-1})\in \{\pm\frac{1}{2}\}$. 
By Proposition \ref{prop:geometricrecursion} the following recursion relations are satisfied:
\begin{equation*}
 \sum_{\delta} k(\c_n,\delta) \brunit{\Ga}{\c_n+\delta}=Z(\c_n,\delta(e_0),\delta(e_{p-1}))\brunit{\Ga}{\c_n}
\end{equation*}
where the sum is taken over all the admissible perturbations $\delta$ of $\c_n$ around the face $D$ of $P$ bounded by $\lambda$ and whose values on $e_0$ and $e_{p-1}$ are $\delta(e_0),\delta(e_{p-1})$ (see Definition \ref{def:admperturb}). 
By letting $\delta(e_0)$ and $\delta(e_{p-1})$ range over $\{\pm \frac{1}{2}\}$ the above equation can be viewed as a single, $2\times 2$-matrix valued equation.
By Assumption \ref{assum:1}, if we divide the equation by $\brunit{\Ga}{\c_n}$, evaluate at $A=\exp(\frac{\I\pi}{2n})$ (using Lemma \ref{lem:matrices}) and take the limit $n\to \infty$, then we get:
\begin{equation*}
 \sum_{\delta} \lim_{n\to \infty} \ev_n(k(\c_n,\delta)) \exp{(2\delta_i \frac{\partial f}{\partial \gamma_i})}= \lim_{n\to \infty} \ev_n(Z(\c_n,\delta(e_0),\delta(e_{p-1})).
\end{equation*}
By Lemma \ref{lem:Sasymp}, the right hand side can be recognized as $RC(\overrightarrow{v_{p-1}v_0})R^{-1}U$.

Using Proposition \ref{prop:kasymp} we realize that the above equality can be expressed as:
\begin{align}\label{eq:recursion}
Q(\overrightarrow{v_{2p-2}v_{2p-1}}) \prod_{i=0}^{p-2}  RC(\overrightarrow{v_{2i+1}v_{2i+2}})R^{-1}UQ(\overrightarrow{v_{2i}v_{2i+1}})= RC(\overrightarrow{v_{2p-1}v_0})R^{-1}U.
\end{align}
Now observe that a straightforward computation shows the following: 
$$R^{-1}UQ(\overrightarrow{v_{2i}v_{2i+1}})R=MC(\overrightarrow{v_{2i}v_{2i+1}})M.$$
Replacing this in \eqref{eq:recursion} and simplifying we get 
$$
MC(\overrightarrow{v_{2p-2}v_{2p-1}})M\prod_{i=0}^{p-2} C(\overrightarrow{v_{2i+1}v_{2i+2}})MC(\overrightarrow{v_{2i}v_{2i+1}})M=R^{-1}URC(\overrightarrow{v_{2p-1}v_0})R^{-1}UR.$$
Finally a last computation shows that 
$$
R^{-1}URC(\overrightarrow{v_{2p-1}v_0})R^{-1}UR=-\left( C(\overrightarrow{v_{2p-1}v_0})\right)^{-1}
$$
and thus we conclude by multiplying on the right by $C(\overrightarrow{v_{2p-1}v_0})$.
The fact that the same equation is satisfied by $\vol(\Ptr)$ is a consequence of the Schl\"afli formula and of Theorem  \ref{teo:volsatisfies}.
\end{proof}

\appendix
\section{Proof of Theorem \ref{teo:easycase}}\label{sec:easycase}
\subsection{Volumes of hyperideal hyperbolic tetrahedra}\label{sub:muya}
In \cite{MuYa} J. Murakami and M. Yano found a formula for the hyperbolic volume of a hyperbolic compact tetrahedron and the formula was later shown by A. Ushijima to hold also for truncated tetrahedra (i.e.\ the compact polyhedra obtained by truncating hyperideal tetrahedra as explained in Subsection \ref{sub:conjecture}). We now recall this formula.

With the notation of Example \ref{ex:tetrahedron} for the dihedral angles of a tetrahedron, let $A,B,C$ and $A',B',C'$ be respectively $-\exp(-\I\alpha),-\exp(-\I\beta),-\exp(-\I\gamma)$ and $-\exp(-\I\alpha')$, $-\exp(-\I\beta')$, $-\exp(-\I\gamma')$ \footnote{In \cite{MuYa} internal dihedral angles are used, this accounts for the change in the definition of $A,B,C,A',B',C'$}.  Let $\mathrm{Li}_2(z):=\int_0^z \frac{\log(1-t)}{-t}\mathrm{d}t=\sum_{n>0}\frac{z^n}{n^2}$ be the dilogarithm function (well-defined on $\C\setminus [1,\infty)$) and $\Lambda(x):=\int_0^x -\log |2\sin t|\, \mathrm{d}t$. Define
\begin{align}\label{eq:UMuYa}
U(z):=\frac{1}{2}\big(\mathrm{Li}_2(z)+\mathrm{Li}_2(zABA'B')+\mathrm{Li}_2(zACA'C')+\mathrm{Li}_2(zBCB'C')\\ \nonumber -\mathrm{Li}_2(-zABC)-\mathrm{Li}_2(-zAB'C')-\mathrm{Li}_2(-zA'BC')-\mathrm{Li}_2(-zA'B'C) \big)
\end{align}
\begin{eqnarray}\label{eq:deltaMuYa}
\Delta(x,y,z) &:=& \frac{-1}{4}\big(\mathrm{Li}_2(-\frac{xy}{z})+\mathrm{Li}_2(-\frac{yz}{x})+\mathrm{Li}_2(-\frac{zx}{y})\\
\notag && +\mathrm{Li}_2(-\frac{1}{xyz})+(\log(x))^2+(\log(y))^2+(\log(z))^2   \big)
\end{eqnarray}
\begin{eqnarray*}
V(z)&:=&\Delta(A,B,C)+\Delta(A,B',C')+\Delta(A',B,C')+\Delta(A',B',C)\\ \notag && +\frac{1}{2}(\log(A)\log(A')+\log(B)\log(B')+\log(C)\log(C'))+U(z).
\end{eqnarray*} 
For later purposes, recall also that for all $\theta\in (0,2\pi)$ one has $\Im(\frac{\mathrm{Li}_2( \exp{\I\theta})}{2})=\Lambda(\frac{\theta}{2})$,
so replacing the values of $A,B,C,A',B',C'$ and setting $z=\exp(\I s)$ in \eqref{eq:UMuYa} we get
\begin{align}\label{eq:imaginaryu}
\Im(U(\exp(\I s)))=\Lambda\left(\frac{s}{2}\right)+\Lambda\left(\frac{s-(\alpha+\alpha'+\beta+\beta')}{2}\right)\\
 \nonumber+ \Lambda\left(\frac{s-(\alpha+\alpha'+\gamma+\gamma')}{2}\right)+\Lambda\left(\frac{s-(\gamma+\gamma'+\beta+\beta')}{2}\right)-\Lambda\left(\frac{s-(\alpha+\beta+\gamma)}{2}\right)\\
 \nonumber -\Lambda\left(\frac{s-(\alpha+\beta'+\gamma')}{2}\right)-\Lambda\left(\frac{s-(\alpha'+\beta+\gamma')}{2}\right)-\Lambda\left(\frac{s-(\alpha'+\beta'+\gamma)}{2}\right).
\end{align}

Let now $z_\pm$ be the two nontrivial solutions of the equation (which reduces to a degree $2$ polynomial equation):
$$ \frac{\mathrm{d}}{\mathrm{d}z} U(z)\in \frac{\pi \I }{z} \mathbb{Z}.$$
As shown by Ushijima \cite{Us} these can be expressed as:
\begin{equation}\label{eq:formulazpm}
\textstyle z_\pm= -2\frac{\sin(\alpha)\sin(\alpha')+\sin(\beta)\sin(\beta')+\sin(\gamma)\sin(\gamma')\pm\sqrt{\det(G)}}{AA'+BB'+CC'+ABC'+A'BC+AB'C+A'B'C'+ABCA'B'C'}\end{equation}
where
\begin{equation}\label{eq:gram}
G=\begin{pmatrix}
1 & \cos(\alpha) & \cos(\beta) &\cos(\gamma')\\
\cos(\alpha) & 1 &  \cos(\gamma) &\cos(\beta')\\
\cos(\beta) & \cos(\gamma) & 1 & \cos(\alpha')\\
\cos(\gamma') & \cos(\beta') &  \cos(\alpha') &1
\end{pmatrix}\end{equation}
Then the following was first proved by Murakami and Yano in \cite{MuYa} for compact hyperbolic tetrahedra and was later shown by A. Ushijima \cite{Us} to hold for the case of a truncated hyperideal tetrahedron:
\begin{theorem}[\cite{MuYa},\cite{Us}]\label{teo:muya}
The volume of the truncated hyperbolic tetrahedron whose exterior dihedral angles are as in Example \ref{ex:tetrahedron}, is $$\vol(Tet)=\Im\left(V(z_-)\right)=-\Im\left(V(z_+)\right)=\Im\left(\frac{U(z_-)-U(z_+)}{2}\right).$$
\end{theorem}
\subsection{Proof of Theorem \ref{teo:easycase}}
We will need the following:
\begin{lemma}\label{lem:factasymp}
Let $\alpha\in (0,1)$ and let $(a_n)_{n\in \mathbb{N}}$ be a sequence of integers such that $\lim_{n\to \infty} \frac{a_n}{n}=\alpha$.
Then $\{a_n\}$ is meromorphic, for $n$ sufficiently large it has no pole at $\exp\left(\frac{\I\pi}{2n}\right)$ and the following holds:
$$\lim_{n\to \infty} \frac{\pi}{n}\log\left (\frac{\ev_n(\{a_n\}!)}{\I^{a_n}}\right )=-\Lambda(\pi \alpha)$$
where we imply that the argument of the $\log$ is real positive.

Similarly, if $\alpha\in (1,2)$ and $(a_n)_{n\in \mathbb{N}}$ is a sequence of half-integers such that $\lim_{n\to \infty} \frac{a_n}{n}=\alpha$, then $\{a_n\}!$ has a simple zero at $\exp\left(\frac{\I\pi}{2n}\right)$ for $n$ big enough and the following holds:
$$\lim_{n\to \infty} \frac{\pi}{n}\log\left(\frac{\ev_n(\{a_n\}!)}{(-1)^{n+a_n}2n^2 \I^{a_n+1}\exp(-\frac{\I\pi}{2n})}\right)=-\Lambda(\pi \alpha)$$
where, again, we imply that the argument of the $\log$ is real positive.
\end{lemma}
\begin{proof}
We limit ourselves to a sketch, see \cite{Co} for a detailed proof. Clearly $\ev_n(fg)=\ev_n(f)\ev_n(g)$, and if $k$ is not a multiple of $n$,  $\ev_n(\{k\})=2\I\sin(\frac{\pi k}{n})$  while $\ev_n(\{n\})=-2n\exp(\frac{-\I\pi}{2n})$ and thus if $a_n\in [0,n)$ (which is true for the first statement),
$$\ev_n(\{a_n\}!)=\prod_{k=1}^{a_n} 2\I\sin \frac{\pi k}{n} \sim \I^{a_n}\exp \frac{-n \left (\Lambda  (\frac{\pi a_n}{n} )-\Lambda (\frac{\pi}{n}  ) \right )}{\pi};$$
while if $a_n\in [n,2n)$,
\begin{eqnarray*}\ev_n(\{a_n\}!) &=& \left (\prod_{k=1}^{n-1} 2\I\sin \frac{\pi k}{n} \right) \left (-2n\exp \frac{-\I\pi}{2n} \right ) \left ( \prod_{k=n+1}^{a_n} 2\I\sin \frac{\pi k}{n} \right ) \\
&=& \left ( -2\I^{n-1}n^2\exp \frac{-\I\pi}{2n} \right ) \prod_{k=n+1}^{a_n} -\I\left | 2 \sin \frac{\pi k}{n}  \right | \\ 
&\sim& -(-\I)^{a_n+1}(-1)^n2n^2\exp \frac{-\I\pi}{2n} \exp\frac{-n \left (\Lambda (\frac{\pi a_n}{n})-\Lambda(\pi+ \frac{\pi}{n})\right )}{\pi}.\end{eqnarray*}
\end{proof}

We can now prove Theorem \ref{teo:easycase}.
For the sake of self-containedness  
we start by sketching the proof of the following proved in \cite{Co} for the skein normalization of the tetrahedron:
\begin{theorem}\label{teo:tetcase}
Let $\Ga$ be the $1$-skeleton of a hyperideal tetrahedron $Tet$ whose exterior dihedral angles are as in Example \ref{ex:tetrahedron} and let $(\c_n)_{n\in \mathbb{N}}$ be a sequence of colorings on $\Ga$ such that $\c_n\in [\frac{n}{2},n)$ and $\lim_{n\to \infty} 2\pi(1-\frac{c_n}{n})$ equals the corresponding exterior dihedral angles of $Tet$. Then,
$$\lim_{n\to \infty}\frac{\pi}{n}\log\big(| \ev_n \brunit{\Ga}{\c_n}|\big)=\vol(Tet^{\mathrm{trunc}}).$$
and Conjecture \ref{conj:vc} holds in this case.
 \end{theorem}
\begin{proof}
The idea of the proof is to first identify the leading term in \eqref{eq:tet}, then to recognize it as the volume of the tetrahedron using Theorem \ref{teo:muya}. Despite the signs present in \eqref{eq:tet}, there will be no cancellation.
Let us start by computing the evaluation at $A=\exp(\frac{\I\pi}{2n})$ of $\Delta(a,b,c)$ (as in \eqref{eq:tet}) using Lemma \ref{lem:factasymp}:
\begin{align*}
\lim_{n\to \infty}\frac{\pi}{n}\log(|\ev_n(\sqrt{\Delta(a,b,c)})|)=\frac{1}{2}\left ( \Lambda(\frac{\alpha+\beta-\gamma}{2})+ \right . \\
\left . \Lambda(\frac{\alpha-\beta+\gamma}{2})+\Lambda(\frac{-\alpha+\beta+\gamma}{2})-\Lambda(\frac{\alpha+\beta+\gamma}{2}) \right )
\end{align*}
and similarly for $\Delta(a,e,f),\Delta(d,b,f),\Delta(d,e,c)$.
Indeed for instance we have $\lim_{n\to \infty} \frac{a_n}{n}=1-\frac{\alpha}{2\pi}$ (and similarly for the other ratios) and so $$\lim_{n\to \infty} \frac{\pi}{n} \log(\ev_n([a_n+b_c-c_n]!))=-\Lambda(\pi-\frac{\alpha+\beta-\gamma}{2})=\Lambda(\frac{\alpha+\beta-\gamma}{2}).$$
Observe that the above formula equals the imaginary part of Equation \eqref{eq:deltaMuYa} (recall that for all $\theta\in (0,2\pi)$ one has $\Im(\frac{\mathrm{Li}_2( \exp{\I\theta})}{2})=\Lambda(\frac{\theta}{2})$).

Now let us concentrate on the summation in Formula \eqref{eq:tet}. Let $S_k$ be the $k^{th}$ summand and remark that $k$ ranges in $[{\rm max}\{T_i\}, {\rm min}\{Q_j\}]$. We claim that $S_k$ has a simple zero at $A=\exp\left(\frac{\I\pi}{2n}\right)$ if $k\in [{\rm max}\{T_i\},n-2]$ and a zero of higher order otherwise, so only the summands $S_k$ with $k\in [{\rm max}\{T_i\},n-2]$ need be considered for the purpose of computing $\ev_n$. 

Indeed by the inequalities of Example \ref{ex:tet} translated in terms of $\gamma\sim 2\pi (1-\frac{\c_n}{n})$, for $n$ large enough, using the notation of Definition \ref{ex:tet}, one has $n<\max(\{T_i\})< 2n$; moreover by hypothesis $\c(e_i)\in [\frac{n}{2},n)$ for every $e_i\in E(\Ga)$ so $\c(e_i)+\c(e_j)-\c(e_k)< n$ for every $3$-uple of edges touching a common vertex (in whatever order). Since the summation index $k$ ranges from $\max(\{T_i\})$ to $\min(\{Q_j\})$ the differences $k-T_i$ and $Q_j-k$ are bounded above by a term of the form $\c(e_i)+\c(e_j)-\c(e_k)< n$ (for some triple of edges sharing a vertex) and hence 
the quantum factorials in the denominator of the quantum binomials forming the summands in Formula \eqref{eq:tet} have arguments $<n$ and are nonzero at $A=\exp\left(\frac{\I\pi}{2n}\right)$. By contrast, the numerator has a simple zero $\exp\left(\frac{\I\pi}{2n}\right)$ when $k$ ranges in $[\max (\{T_i\}),n-2]$ and a double zero when $k$ ranges in $[n-1,\min(\{Q_j\})]$. The claim is thus proved. 

Our second claim is now that for all $k\in [\max (\{T_i\}),n-3)$, the ratio $\ev_n\big(\frac{S_{k+1}}{S_k}\big)$ is real positive.
Indeed:
$$ \frac{S_{k+1}}{S_k}=-\frac{\{k+1\}\{Q^{(n)}_1-k\}\{Q^{(n)}_2-k\}\{Q^{(n)}_3-k\}}{\{k+1-T^{(n)}_1\}\{k+1-T^{(n)}_2\}\{k+1-T^{(n)}_2\}\{k+1-T^{(n)}_4\}}$$
where we let $Q^{(n)}_j$ and $T^{(n)}_i$ be the ``squares and triangles'' associated to the coloring $\c_n$ as in Definition \ref{ex:tet}. Since $\{k\}=2\I\sin(\frac{\pi k}{n})$, by the same estimates as in the previous claim the ratio is then positive real as $\{k+1\}$ is a negative multiple of $2\I$.  So ${\rm max}\{|S_k|\}\leq |\ev_n\big(\sum_k S_k)|\leq (n-2-{\rm max}\{T_i\}) {\rm max}\{ |S_k|\}$, and we are left to find the terms $S_k$ for which $|S_k|$ is maximal.

Now remark that since the arguments of the factorials in the denominators of $S_k$ all belong to $[0,n)$, their evaluations (by Lemma \ref{lem:factasymp}) grow respectively like $\I^{k-T_i}\exp(-\frac{n}{\pi}\Lambda(\pi \frac{k-T_i}{n}))$ and $\I^{Q_j-k}\exp(-\frac{n}{\pi}\Lambda(\pi \frac{Q_j-k}{n}))$.
By contrast the numerator grows like $$\I^{k+2}(-1)^{k+1+n}2n^2\exp\left(\frac{-\I\pi}{2n}\right)\exp\left(-\frac{\pi}{n}\Lambda\left(\pi\frac{k}{n}\right)\right),$$ so, taking into account the sign $(-1)^k$ in front of the multinomial, the $k^{th}$-summand grows like:
$$\frac{(-1)^n2n^2 \exp\left(-\frac{\I\pi}{2n}\right)}{\{1\}} \exp\left (\frac{n}{\pi}\left (-\Lambda\left(\pi\frac{k}{n}\right)+\sum_i\Lambda\left(\pi\frac{k-T_i}{n}\right)+\sum_j \Lambda\left(\pi\frac{Q_j-k}{n}\right)\right ) \right ).$$
Now let us define numbers $\tau_i, \nu_j$ by 
$$\lim_{n\to \infty}\pi \frac{T^{(n)}_i}{n}=3\pi-\frac{\ga(e)+\ga(e')+\ga(e'')}{2}
=:3\pi-\frac{\tau_i}{2}$$ 
$$\lim_{n\to \infty}\pi \frac{Q^{(n)}_j}{n}=4\pi-\frac{\ga(e)+\ga(e')+\ga(e'')+\ga(e''')}{2}
=:4\pi -\frac{\nu_j}{2}$$
for the edges $e,e',e'',e'''$ 
naturally associated to $T_i$ or $Q_j$. 
Letting $\pi\frac{k}{n}=2\pi-\frac{s}{2}$ and $$f(s)=-\Lambda\left(2\pi-\frac{s}{2}\right)+\sum_i\Lambda\left(\frac{-s+\tau_i}{2}-\pi\right)-\sum_j \Lambda\left(-2\pi+\frac{\nu_j-s}{2}\right)$$ and comparing with Equation \eqref{eq:imaginaryu}, we see that $f(s)=\Im \left(U\left(\exp\left(\I s\right)\right)\right)$ and that the growth rate of the logarithm of the norm of the evaluation of the sum is given by: $$\max_{s\in [0,\min(\{\tau_i-2\pi\})]} \frac{n}{\pi}f(s).$$
We thus search for the maximum of $f$ and imposing $f'(s)=0$ we get an equation of the form:
$$\frac{|\sin\left(2\pi-\frac{s}{2}\right)\sin\left(\frac{\nu_1-s}{2}\right)\sin\left(\frac{\nu_2-s}{2}\right)\sin\left(\frac{\nu_3-s}{2}\right)|}{|\sin\left(\frac{\tau_1-s-2\pi}{2}\right)\sin\left(\frac{\tau_2-s-2\pi}{2}\right)\sin\left(\frac{\tau_3-s-2\pi}{2}\right)\sin\left(\frac{\tau_4-s-2\pi}{2}\right)|}=1
$$
Recalling that $\nu_j-\tau_i>0$ for all $i,j$ (see Example \ref{ex:tet}), the above equation is equivalent to:
$$\frac{\sin\left(2\pi-\frac{s}{2}\right)\sin\left(\frac{\nu_1-s}{2}\right)\sin\left(\frac{\nu_2-s}{2}\right)\sin\left(\frac{\nu_3-s}{2}\right)}{\sin\left(\frac{\tau_1-s-2\pi}{2}\right)\sin\left(\frac{\tau_2-s-2\pi}{2}\right)\sin\left(\frac{\tau_3-s-2\pi}{2}\right)\sin\left(\frac{\tau_4-s-2\pi}{2}\right)}=1
$$
Replacing $\sin(x)$ by $\frac{e^{\I x}-e^{-\I x}}{2\I}$ and setting $A:=-\exp\left(-\I\alpha\right)$ (and similarly for all the angles) and $z=-\exp(\I s)$ one gets a polynomial equation of degree $2$ in $z$ whose solutions were shown by Murakami and Yano to be $z_+,z_-$ given in Formula \eqref{eq:formulazpm}. (Even if this is not apparent from the formula, as soon as the Gram matrix in Equation \eqref{eq:gram} has negative determinant, then $|z_-|=|z_+|=1$.) In particular, in the case of a regular ideal octahedron (i.e.\ a maximally truncated tetrahedron), whose exterior dihedral angles are all $\pi$, $f(s)=4\Lambda\left(\frac{s}{2}\right)+4\Lambda\left(\frac{\pi}{2}-\frac{s}{2}\right),\ s\in [0,\pi]$, whose extremum is attained at $s=\frac{\pi}{2}$ and its value is $8\Lambda\left(\frac{\pi}{4}\right)>0$.
By Theorem \ref{teo:muya} the corresponding point is then $z_-$ because $\vol (Tet)=\Im\left(V\left(z_-\right)\right)=-\Im\left(V\left(z_+\right)\right)=\Im\left(\frac{U\left(z_-\right)-U\left(z_+\right)}{2}\right)>0$.

One concludes the proof by putting together all the terms in the asymptotical behavior of Formula \eqref{eq:tet} and comparing with Theorem \ref{teo:muya}. 
\end{proof}

We are now able to prove Theorem \ref{teo:easycase} in full generality.
We do this by induction on the number $n$ of moves of the form  \raisebox{-0.1cm}{\includegraphics[width=1.0cm]{vertextotriangle.pdf}} producing $\Ga$ from a tetrahedron. If $n=0$ then $\Ga$ is the $1$-skeleton of a hyperideal hyperbolic tetrahedron whose possible angles are provided by Example \ref{ex:tet}. For such a graph, Theorem \ref{teo:tetcase} proves Theorem \ref{teo:easycase}.
If now $\Ga$ can be obtained from $\Ga'$ (for which we assume Theorem \ref{teo:easycase} holds) by a single move, suppose that $\alpha,\beta,\gamma$ are the exterior angles on the three edges of $\Ga'$ involved in the move. Observe that any geometric structure of the hyperbolic polyhedron $\Ptr_{\Ga}$ can be obtained by gluing two polyhedra $\Ptr_{\Ga'}$ and $Tet^{\mathrm{trunc}}$, along the face corresponding to the triple of edges of $\Ga'$ where the move is applied.  Conversely if $\Ptr_{\Ga'}$ and $Tet^{\mathrm{trunc}}$ are equipped with geometric structures such that the dihedral angles along the truncation faces to be matched are the same, then they can be glued to form a hyperideal hyperbolic structure on $\Ptr_{\Ga}$.
Clearly the volumes add up: $\vol(\Ptr_{\Ga})=\vol(\Ptr_{\Ga'})+\vol(Tet^{\mathrm{trunc}})$.

On the quantum side, this follows from equation \eqref{eq:triangleformula}. For any coloring $\c_n$ on $\Ga$ one has $$\brunit{\Ga}{\c_n}=\brunit{\Ga'}{\c'_n}\brunit{Tet}{\c''_n}$$
where $\c'_n$ and $\c''_n$ are the restrictions of $\c_n$ to the edges of $\Ga'$ and $Tet$ respectively (here we use the natural injections $E(\Ga')\to E(\Ga)$ and $E(Tet)\to E(\Ga)$ to restrict the $\c_n$). 
Then one concludes by induction using Theorem \ref{teo:tetcase}.\qed

\bibliography{biblio}{}
\bibliographystyle{plain}

\end{document}